\documentclass[11pt,a4paper]{article}
\usepackage[a4paper,left=2.5cm,right=2.5cm,top=2.5cm,bottom=2.5cm]{geometry} 
 
 \usepackage{amsmath,amssymb,amsthm,amsfonts,subcaption,caption}
\usepackage[colorlinks=True,
            linkcolor=blue,
            anchorcolor=green,
            citecolor=blue
            ]{hyperref} 
\usepackage[noabbrev,capitalise,nameinlink]{cleveref}
\usepackage{float}
\usepackage{graphicx}
\usepackage{epstopdf}
\usepackage{multirow} 
\usepackage{enumitem,url}
\usepackage{threeparttable}  

\usepackage[ruled,vlined]{algorithm2e}
\usepackage{etoolbox}

 \newtheorem{theorem}{Theorem}[section]

\newtheorem{lemma}{Lemma}[section]

\newtheorem{definition}{Definition}[section]
\newtheorem{remark}{Remark}[section]

\newtheorem{example}{Example}[section]

\usepackage{hyperref}
\newcommand{\footremember}[2]{%
    \footnote{#2}
    \newcounter{#1}
    \setcounter{#1}{\value{footnote}}%
}

\graphicspath{ {./Figures/} }

\def\B{{\mathbb B}}

\def\R{{\mathbb  R}}

\def\C{\mathbb{C}}
\def\H{{\bf  H}}

\def\Q{{\bf  Q}}

\def\v{{\bf v}} 
\def\w{{\bf w}}
\def\z{{\bf z}}
\def\x{{\bf x}}

\def\y{{\bf y}}

\author{
Shuai Li\footremember{bjtu1}{School of Mathematics and Statistics, Beijing Jiaotong University, China.}~~
    Shenglong Zhou\footnotemark[2]
  \footremember{bjtu2}{Email: 24110488@bjtu.edu.cn, shlzhou@bjtu.edu.cn}
}

\title{\vspace{-1.25cm}
\bf Computing Binary Integer Programming \\
via A New Exact Penalty Function\thanks{This work is supported by the Fundamental Research Funds for the Central Universities (2024YJS091).}
\vspace{-0.25cm}}

\date{}

\begin{document}
\flushbottom
 
\maketitle
 
\vspace{-1.3cm}

\begin{abstract}  
\noindent \textbf{Abstract:} Unconstrained binary integer programming (UBIP) poses significant computational challenges due to its discrete nature. We introduce a novel reformulation approach using a piecewise cubic function that transforms binary constraints into continuous equality constraints. Instead of solving the resulting constrained problem directly, we develop an exact penalty framework with a key theoretical advantage: the penalty parameter threshold ensuring exact equivalence is independent of the unknown solution set, unlike classical exact penalty theory. To facilitate the analysis of the penalty model, we introduce the concept of P-stationary points and systematically characterize their optimality properties and relationships with local and global minimizers. The P-stationary point enables the development of an efficient algorithm called APPA, which is guaranteed to converge to a P-stationary point within a finite number of iterations under a single mild assumption, namely, strong smoothness of the objective function over the unit box. Comprehensive numerical experiments demonstrate that APPA outperforms established solvers in both accuracy and efficiency across diverse problem instances.
\vspace{0.3cm} 
 
\noindent{\textbf{Keywords}:} UBIP, piecewise cubic function, exact penalty theory,  P-stationary points, global convergence, termination within finite iterations
\end{abstract}

\numberwithin{equation}{section}

\section{Introduction}\label{Section-Introduction}
This paper focus on the following unconstrained binary integer programming (UBIP), 
\begin{equation}
\label{UBIP}\min\limits_{\x\in\R^n}~f(\x),~~{\rm s.t.}~~\x\in\{0,1\}^n,\tag{UBIP}
\end{equation}
where function ${f:\R^n\to\R\cup\{\infty\}}$ is continuously differentiable. The UBIP problem finds extensive applications spanning traditional combinatorial optimization and contemporary machine learning. In combinatorial optimization, it models classical problems like the knapsack problem \cite{bas23,mer78} and the max-cut problem \cite{bur02,yang22}. In machine learning, UBIP has attracted considerable attention for adversarial attacks \cite{esm19,wee21}, transductive inference \cite{joa99}, and binary neural networks \cite{las16,qin20}. For additional applications, we refer readers to \cite{kochenberger2014unconstrained,liu23,pan23}. Despite its broad applicability, \eqref{UBIP} is known to be NP-hard due to the combinatorial nature of binary variables \cite{gar79,koc14}. Instead of directly solving \eqref{UBIP},  we aim to address the following constrained optimization, 
\begin{equation}\label{SCO}
\begin{aligned}
\min\limits_{x\in{\R^n}}~ f(\x),~~{\rm s.t.}~~& g(x_i) =0,~x_i\in B,~i\in[n], 
\end{aligned} 
\end{equation}
where $[n]:=\{1,2,\cdots,n\}, B:= \{x\in\R: 0\leq x \leq 1\},$ and $g:\R\to\R\cup\{\infty\}$ is a piecewise cubic function defined as follows and its graph is shown in Figure \ref{fig:cubic},
\begin{equation}\label{def-h-h}
g(x):=\begin{cases} x^3-3x^2+3x,&x \leq  1/2,\\
1-x^3 ,&x>  1/2.
\end{cases} 
\end{equation} 
 
\begin{figure}[!t]
	\centering
		\includegraphics[width=0.5\textwidth]{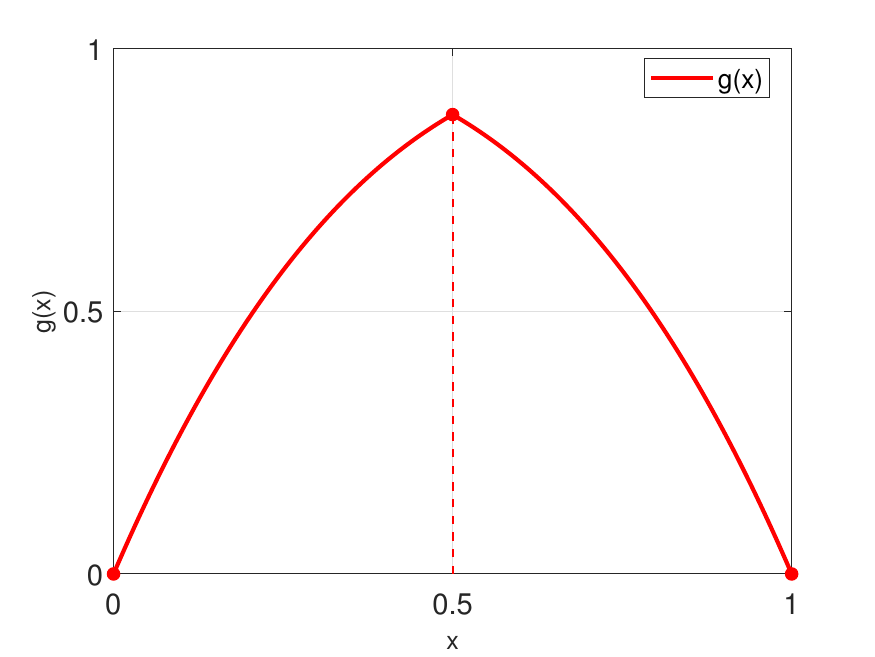}
	\caption{The graph of $g(x)$} \label{fig:cubic}
\end{figure}

\subsection{Related work}  

Various approaches have been developed to process binary constraints in optimization problems, which can be broadly classified into two categories: relaxation methods and equivalent reformulations. Table \ref{setupdiff} provides a summary of representative techniques from each category, which we briefly review below.

Relaxation approaches address binary constraints by transforming them into continuous optimization models. A widely adopted approach is linear programming (LP) relaxation \cite{Hsieh15, Kom07}, which replaces the discrete binary constraint with a continuous box constraint. This reformulation converts the NP-hard binary problem into a tractable convex optimization problem that can be efficiently solved by modern first-order methods. Spectral relaxation \cite{Cour07, Shi07} offers an alternative strategy by substituting the binary constraint with a spherical constraint, allowing the problem to be solved via eigenvalue decomposition.  Semi-definite programming (SDP) relaxation provides another powerful framework by introducing a matrix lifting technique and relaxing the binary constraint through a positive semidefinite cone \cite{wolkowicz2012handbook, anjos2011handbook}. The resulting SDP formulations can be solved using various specialized algorithms, including interior point methods \cite{wolkowicz2012handbook}, quasi-Newton and smoothing Newton methods \cite{Wang16}, spectral bundle methods \cite{helmberg2000spectral}, and branch-and-bound algorithms \cite{buchheim2013semidefinite, buc19}. To achieve tighter relaxations, doubly positive methods \cite{Huang14, Wen10} impose non-negativity constraints on both the eigenvalues and entries of the SDP solution matrix. Empirical studies have shown that these refined relaxation techniques can substantially improve solution quality compared to  SDP approaches.
 
 Equivalent reformulations provide another approach to addressing binary constraints in optimization problems. Unlike relaxation methods, these approaches transform the problem while preserving exact equivalence. A piecewise separable reformulation \cite{Zhang07} converts UBIP into an equality-constrained optimization problem solvable by penalty methods. Continuous $\ell_2$-box non-separable reformulation \cite{Rag69} has been employed to reformulate the UBIP as a continuous optimization problem, enabling the use of second-order interior-point methods \cite{Mar12, Murray10}. More generally, \cite{Wu19} proposed a continuous $\ell_p$-norm ($p > 0$) box reformulation and applied ADMM for its solution. A different perspective was taken in \cite{Yuan17}, where the authors reformulated UBIP as an MPEC problem by embedding the $\ell_2$ norm constraint within an equilibrium framework. The resulting formulation is an augmented biconvex optimization problem with a bilinear equality constraint, amenable to solution via exact penalty methods. Additionally, binary optimization can be reformulated as an $\ell_0$ norm semi-continuous optimization problem \cite{yuan2016sparsity}. Our proposed approach based on piecewise cubic function also belongs to this class of equivalent reformulations, as shown in Table \ref{setupdiff}.

 \begin{table}[!t]
\renewcommand{\arraystretch}{1.1}\addtolength{\tabcolsep}{0pt}
\centering
\caption{Main techniques to tackle the binary constraints. \label{setupdiff}}
\begin{tabular}{lll}\hline
 Methods & References &Reformulations \\
\hline
\multicolumn{3}{c}{Relaxation methods}\\\hline
 LP relaxation &\cite{Hsieh15,Kom07}& $\{0,1\}^n \approx\{\x\mid 0 \leq \x \leq 1\}$ \\

 Spectral relaxation &\cite{Cour07,Shi07} & $\{0,1\}^n \approx\left\{\x \mid\|2\x-1\|_2^2=n\right\}$\\

 SDP relaxation &\cite{buc19,Wang16}  & $\{0,1\}^n \approx\left\{\x \mid \textbf{X} \succeq \x \x^\top, \operatorname{diag}( \textbf{X} )=\x\right\}$ \\

Doubly positive relaxation &\cite{Huang14,Wen10} & $\{0,1\}^n \approx\left\{\x \mid  \textbf{X}  \succeq (2\x-1) (2\x-1)^\top, \operatorname{diag}( \textbf{X} )=1\right\}$ \\\hline 
\multicolumn{3}{c}{Equivalent reformulations}\\\hline
 Piecewise &\cite{Zhang07} & $\{0,1\}^n \Leftrightarrow\{\x \mid\x\odot(1-\x)=0\}$ \\ 
 
  $\ell_2$ box &\cite{Murray10,Rag69} & $\{0,1\}^n \Leftrightarrow\left\{\x \mid 0 \leq \x \leq 1,\|2\x-1\|_2^2=n\right\}$ \\

  $\ell_2$ box MPEC  &\cite{Yuan17} & $\{0,1\}^n \Leftrightarrow\left\{\x \mid 0\leq \x \leq1, \|\v\|^2\leq n,\langle 2\x-1,\v\rangle = n\right\}$ \\


  $\ell_p$ box ($p>0$) &\cite{Wu19} & $\{0,1\}^n \Leftrightarrow\left\{\x \mid0\leq \x \leq 1,\|2\x-1\|_p^p=n, p>0\right\}$ \\

$\ell_0$ norm &\cite{yuan2016sparsity} & $\{0,1\}^n \Leftrightarrow\left\{\x \mid\|\x\|_0+\|\x-1\|_0 \leq n\right\}$ \\
Piecewise cubic & This paper & $\{0,1\}^n \Leftrightarrow\left\{\x \mid  0 \leq \x \leq 1, g(x_i)=0,i\in[n]\right\}$ \\
\hline
\end{tabular}
\end{table}
\subsection{Contribution and organization} 
In this paper, we introduce a piecewise cubic function to reformulate UBIP as the continuous model \eqref{SCO}. This reformulation offers two key advantages: strong theoretical guarantees and the ability to develop efficient numerical algorithms. In other words, it provides a practical and effective continuous optimization framework for solving the discrete programming problem \eqref{UBIP}. Our main contributions are summarized as follows.

\textit{1) New exact penalty theorems with solution-independent penalty parameters}.  To tackle \eqref{SCO} (equivalent to \eqref{UBIP}), we focus on its penalty model. By introducing the notion of P-stationary points, which serve as a key tool for both theoretical analysis and algorithm design, we establish novel exact penalty results showing that global minimizers of \eqref{SCO} and the penalty model coincide when the penalty parameter exceeds a threshold that is independent of the solution set of \eqref{SCO}. This property distinguishes our theory from classical exact penalty frameworks, where the penalty parameter typically depends on the solution set of the original problem. Moreover, we systematically analyze the relationships between P-stationary points and local/global minimizers of \eqref{SCO} (see Figure \ref{fig:relation}). These relationships reveal that computing P-stationary points provides a practical and theoretically justified strategy for solving \eqref{UBIP}, inspiring our algorithmic approach.

\textit{2) A simple numerical algorithm with strong convergence properties}. To solve the penalty problem \eqref{pi-penalty}, we develop a novel algorithm called APPA (Adaptive Proximal Point Algorithm). Despite the adaptive updates of both the P-stationary parameter and the penalty parameter, we establish that the whole sequence converges to a P-stationary point when $f$ is strongly smooth on $\B$, as shown in Theorem \ref{global-convergence}. Remarkably, the algorithm exhibits finite termination, as outlined in Theorem \ref{finite-termination}, a property that is stronger than quadratic convergence.

\textit{3) Superior numerical performance}. To evaluate the performance of APPA, we conduct comprehensive numerical experiments across diverse problem classes, benchmarking it against several state-of-the-art solvers, including the commercial solver GUROBI. The numerical experiments demonstrate that APPA achieves competitive solution quality while maintaining significantly faster computational times, especially in high-dimensional settings.

The remainder of this paper is organized as follows. The next subsection introduces the notation and preliminary concepts used throughout. Section \ref{Section-model} presents the penalty reformulation of \eqref{SCO} and establishes the exact penalty theory. Section \ref{Section-algorithm} develops the APPA algorithm and analyzes its convergence properties. Extensive numerical experiments and concluding remarks are given in the last two sections.

\subsection{Preliminaries} \label{Section-pre}
We end this section by introducing some notation to be used throughout the paper. The $n$-dimensional unit box is denoted by
$$ \B:=\{\x\in\R^n:  x_i\in B, i\in[n]\},$$
where $ B:=\{x\in\R:  0\le x\le1\}$, $[n]={1,2,...,n}$, and `$:=$' means `define'. We denote by $\|\x\|$ and $\|\x\|_\infty$ the $\ell_2$- and $\ell_\infty$-norms of a vector $\x\in \R^n$, respectively. For two vectors $\x, \y \in \R^n$, $\langle \x, \y \rangle$ denotes their inner product, i.e., $\langle \x, \y \rangle=\x^{\top}\y=\sum{x_iy_i}$. Let $\lceil t\rceil$ be the ceiling of $t$, i.e., the smallest inter no less than $t$. For a given set $\Omega$, the indicator function $\delta_\Omega(\cdot)$ is defined as
\begin{equation*} 
\delta_\Omega(\x):= \begin{cases} ~0,& {\rm if}~\x\in\Omega,\\
~\infty,& {\rm if}~\x\notin\Omega.\end{cases}
\end{equation*}

We denote by $N_{\B}(\x)$ the normal cone \cite{rock98} of the set $\B$ at the point $\x$. Then any  $\v\in N_{\B}(\x)$ satisfies
\begin{equation}\label{normal-cone-B}
v_i \in \begin{cases} ~(-\infty,0],& {\rm if}~x_i=0,\\
~\{0\},& {\rm if}~x_i\in(0,1),\\
~[0,+\infty),& {\rm if}~x_i=1.\end{cases}
\end{equation}
For a lower semi-continuous function {$f:{\mathbb R}^n\rightarrow \R$}, the definition of the (limiting) subdifferential denoted by {$\partial f$ can be found in \cite[Definition 8.3]{rock98}, which allows us to calculate \begin{equation}\label{g-subdiff}
\partial g(x)=\begin{cases}
\left\{3x^2-6x+3\right\},&x <{1}/{2},\\ 
\left\{- {3}/{4}, {3}/{4}\right\},&x =  {1}/{2},\\ 
\left\{-3x^2\right\} ,&x>  {1}/{2}.
\end{cases} 
\end{equation} 
One can easily verify that
\begin{equation} \label{lowerbound}
	|\nu|\ge3, ~\forall \nu \in \partial g(x), ~\forall x \in B.
\end{equation}
 We say function $f$ is $L$-strong smoothness on box $\B$ if 
$$ 
f(\x) \leq f(\w)+\langle \nabla f(\w), \x-\w\rangle +  \frac{L}{2}\|\x-\w\|^2, ~~\forall~\x,\w\in \B,
$$
where $L>0$. This is a weaker condition than the strong smoothness of $f$ on whole space $\R^n$. A sufficient condition to $L$-strong smoothness on box $\B$ is the twice continuously differentiability of $f$.  Moreover, we say function $f$ is $\ell$-strong convex on  box $\B$ if 
$$ 
f(\x) \geq f(\w)+\langle \nabla f(\w), \x-\w\rangle +  \frac{\ell}{2}\|\x-\w\|^2, ~~\forall~\x,\w\in \B,
$$
where $\ell>0$. The proximal operator of a function ${p}:\R^n\to\R$, associated with a parameter $\tau>0$, is defined by
\begin{equation}\label{proximal-varphi}
{\rm Prox}_{\tau {p}}(\z):=\operatorname*{argmin}\limits_{\x\in \R^n}~ {p}(\x)+\frac{1}{2\tau}\|\x-\z\|^2.
\end{equation}
Specifically, we define the proximal operator of the function $p+\delta_\Omega$ as follows:
\begin{equation}\label{proximal-p-indicator}
	{\rm Prox}_{\tau {p}}^{\Omega}(\z):={\rm Prox}_{\tau ({p}+\delta_\Omega)}(\z)=\operatorname*{argmin}\limits_{\x\in B}~ {p}(\x)+\frac{1}{2\tau}\|\x-\z\|^2,
\end{equation}
where $\Omega$ is a given set. Finally, we present the explicit form of ${\rm Prox}_{\tau g}^B(z)$ in the following lemma, with its graph  shown in Figure \ref{fig:Prox-cubic}.
\begin{figure}[!t]
	\centering
	\begin{subfigure}[b]{0.49\textwidth}
		\centering
		\includegraphics[width=0.99\textwidth]{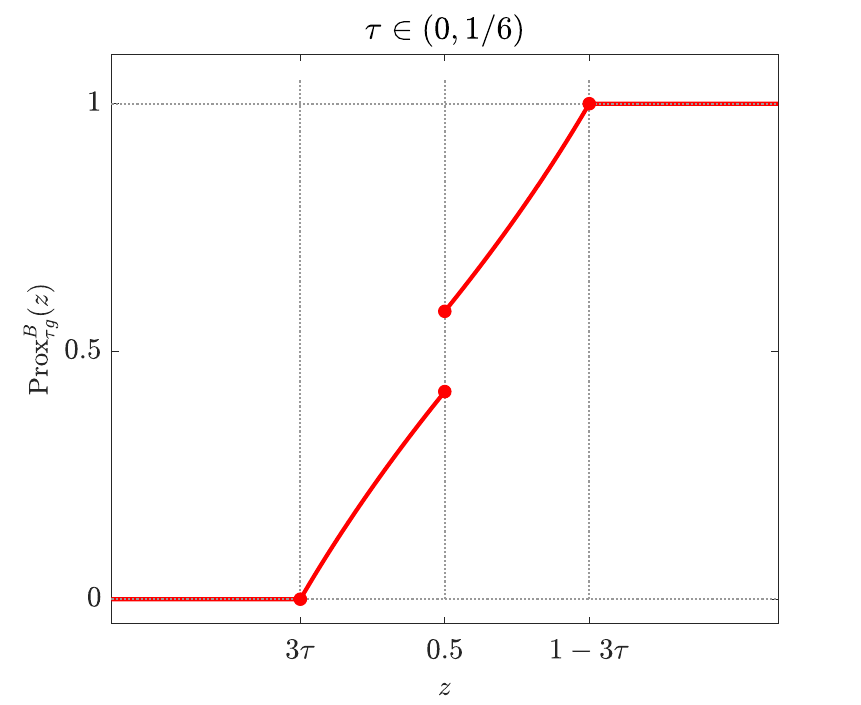}
	\end{subfigure} 
	\begin{subfigure}[b]{0.49\textwidth}
		\centering
		\includegraphics[width=0.99\textwidth]{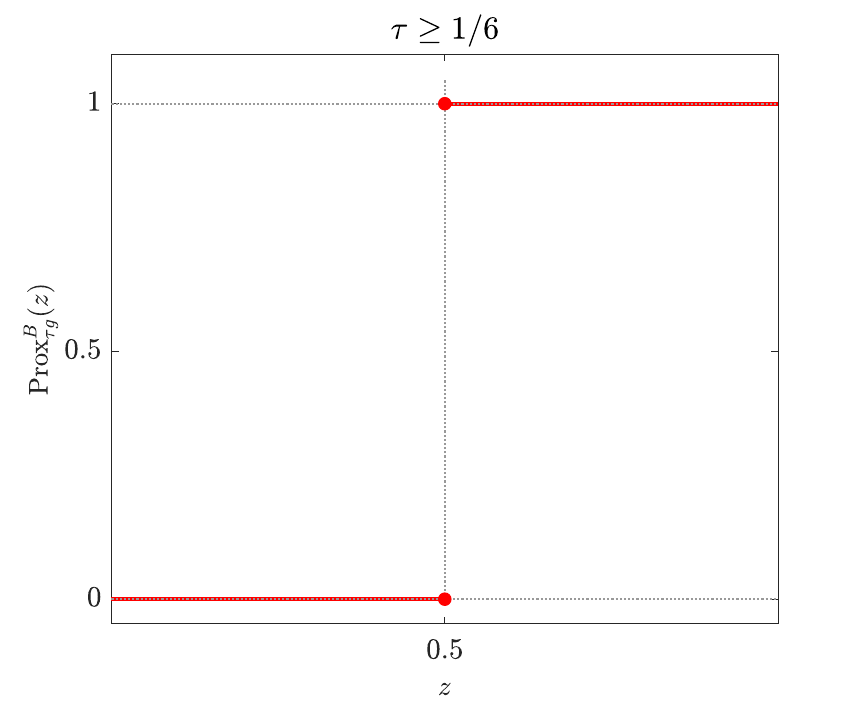}
	\end{subfigure}
		\caption{The graph of ${\rm Prox}_{\tau g}^B(z)$} \label{fig:Prox-cubic}
\end{figure}
\begin{lemma}\label{lemma-proximal-closed}
For any $z \in \R^n$ and $\tau > 0$,  proximal operator ${\rm Prox}_{\tau {g}}^{B}(z)$ takes the following forms.
\begin{itemize}[leftmargin=15pt]
\item If $\tau\geq 1/6$, then
\begin{equation}\label{eq-prox1}
{\rm Prox}_{\tau g}^B(z) = \left\{\begin{array}{ll}
  \left\{0\right\}, &z<{1/2},\\
  \{0, 1\}, &z= {1/2},\\ 
   \left\{1\right\},&z> {1/2}.
 \end{array}\right.
\end{equation}
\item  If $\tau \in(0, 1/6)$, then
\begin{equation}\label{eq-prox2}
{\rm Prox}_{\tau g}^B(z) =\left\{\begin{array}{ll}
  \left\{0\right\}, &z\leq 3\tau,\\[1ex]
  \left\{z_1^*\right\} , &z\in( 3\tau, {1/2}),\\[1ex] 
   \left\{ z_1^* ,  z_2^*  \right\}, &z= {1/2},\\[1ex] 
     \left\{z_2^*\right\}, &z\in ({1/2}, 1-3\tau),\\[1ex] 
    \left\{1\right\},&z\geq 1-3\tau,
 \end{array}\right.\end{equation}
where $z_1^* := 1 + \dfrac{\sqrt{1+12\tau(z-1)}-1}{6\tau}$ and $z_2^* := \dfrac{1-\sqrt{1-12\tau z}}{6\tau}$.
\end{itemize} \end{lemma}

\begin{proof}
	For a given $z \in \mathbb{R}$ and $\eta > 0$, denote $g^* := \operatorname*{min}\limits_{x \in B} g(x) + \dfrac{(x - z)^2}{2\tau}$ and
	\begin{equation*}
		\begin{aligned}
			g_1(x):=&x^3-3x^2+3x+\frac{(x - z)^2}{2\tau},\\
			g_2(x):=&1-x^3+\frac{(x - z)^2}{2\tau}.
		\end{aligned}
	\end{equation*}
	To derive $g^*$, we need to compare the following two optimal objective function values, namely, 
	$$
	g_1^*:=\min_{x \in [0,\frac{1}{2}]} g_1(x), \qquad g_2^*:=\min_{x \in (\frac{1}{2},1]} g_2(x) .
	$$
	Direct calculation enables us to obtain their values under different $\tau$, which are summarized in Table \ref{tab-compare-optvalue}, where $\tau_1:={1/2}-{3\tau/4}$ and $\tau_2:={1/2}+{3\tau/4}$. 	Finally, a comparison of $g_1^*$ and $g_2^*$ in different cases yields \eqref{eq-prox1} and \eqref{eq-prox2}.
\end{proof}
	
	\begin{table}[h]
		\renewcommand{\arraystretch}{1.10}\addtolength{\tabcolsep}{2pt}
		\centering
		\caption{The optimal function values of $g_1$ and $g_2$.}
		\begin{tabular}{cll} 
			\hline
			$\tau$ & \multicolumn{1}{c}{$g_1^*$}  &\multicolumn{1}{c}{$g_2^*$}\\ 
			\hline
			$0<\tau< {1}/{6}$ & $\begin{cases} g_1(0), & z < 3\tau \\ g_1(z_1^*), & z \in [3\tau,\tau_2) \\ g_1({1/2}),\hspace{20mm} & z \ge \tau_2 \end{cases}$ & $\begin{cases} g_2({1/2}),\hspace{20mm} & z < \tau_1 \\ g_2(z_2^*), & z \in [\tau_1,1-3\tau) \\ g_2(1), & z \ge 1-3\tau \end{cases}$ \\
			\hline
			$ {1}/{6}\le\tau< {2}/{9}$ & $\begin{cases} \text{min}\{g_1(0),g_1(z_1^*)\},~ & z < 3\tau \\ g_1(z_1^*), & z \in [3\tau,\tau_2) \\ g_1({1/2}), & z \ge \tau_2 \end{cases}$ & $\begin{cases} g_2({1/2}), & z < \tau_1 \\ g_2(z_2^*), & z \in [\tau_1,1-3\tau) \\ \text{min}\{g_2(1),g_1(z_2^*)\},~ & z \ge 1-3\tau \end{cases}$ \\
			\hline
			$ {2}/{9}\le\tau< {1}/{3}$ & $\begin{cases} \text{min}\{g_1(0),g_1(z_1^*)\}, & z < \tau_2 \\ \text{min}\{g_1(0),g_1({1/2})\}, & z \in [\tau_2,3\tau) \\ g_1({1/2}), & z \ge 3\tau \end{cases}$ & $\begin{cases} g_2({1/2}), & z < 1-3\tau \\ \text{min}\{g_2({1/2}),g_2(1)\}, & z \in [1-3\tau,\tau_1) \\ \text{min}\{g_2(z_2^*),g_2(1)\}, & z \ge \tau_2 \end{cases}$ \\
			\hline
			$\tau\ge {1}/{3}$ & $\begin{cases} g_1(0), & z < \tau_2 \\ \text{min}\{g_1(0),g_1({1/2})\}, & z \in [\tau_2,3\tau) \\ g_1({1/2}), & z \ge 3\tau \end{cases}$ & $\begin{cases} g_2({1/2}), & z < 1-3\tau \\ \text{min}\{g_2({1/2}),g_2(1)\}, & z \in [1-3\tau,\tau_1) \\ g_2(1), & z \ge \tau_2 \end{cases}$ \\ 
			\hline
		\end{tabular}
		\label{tab-compare-optvalue}
	\end{table}

\section{Exact Penalty Theory} \label{Section-model}
Instead of solving problem \eqref{SCO}, we focus on solving its penalty formulation,
\begin{equation}
\label{pi-penalty}
\min_{\x\in \B}~F(\x;{\lambda}):=f(\x)+{\lambda} {p}(\x), \qquad \text{with}~~ {\lambda}\geq\overline{{\lambda}}  ~~\text{and}~~{p}(\x):=\sum_{i=1}^n g(x_i),
\end{equation} 
where $ \overline{{\lambda}}$ is defined by 
\begin{equation}\label{lower-bd-pi}
 \overline{{\lambda}} :=\max_{\x\in\B}\frac{\|\nabla f(\x)\|_\infty}{3}.
\end{equation}
Since $f$ is continuously differentiable and $\B$ is bounded, $\nabla f$ is continuous on $\B$, which implies that $\overline{{\lambda}}$ is well-defined and finite. In the sequel, we show that \eqref{pi-penalty} is an exact penalty model of \eqref{SCO}.

To derive the exact penalty theory, we define a P-stationary point associated with the proximal operator of ${p}$, which also plays a crucial role in our algorithm design.
\begin{definition}
Point $\overline{\x}$ is called a P-stationary point of (\ref{pi-penalty}) if there is a $\tau>0$ such that
\begin{equation}\label{Pstationary}
\begin{aligned}
\overline{\x} &\in {\rm Prox}_{\tau{\lambda} {p}}^\B\Big(\overline{\x} - \tau \nabla f(\overline{\x})\Big),\\
&=  {\rm argmin}_{\z\in\B}~\frac{1}{2}\| \z -(\overline{\x}- \tau \nabla f(\overline{\x}))\|^2+\tau{\lambda} {p}(\z).
\end{aligned}\end{equation} 
\end{definition}
By Lemma \ref{lemma-proximal-closed}, we can directly obtain the closed-form of ${\rm Prox}_{\tau{\lambda} {p}}^\B$, namely,
$${\rm Prox}_{\tau{\lambda} {p}}^\B(\z)=\{\v\in\R^n:v_i\in{\rm Prox}_{\tau{\lambda} {g}}^B(z_i), i\in[n]\},$$ which allows us to design a fast numerical algorithm given in the next section. The following result establishes the relationships among P-stationary points and local minimizers of problem (\ref{pi-penalty}).

 \begin{theorem}\label{firstPST} The following relationships hold for problem (\ref{pi-penalty}).  
\begin{itemize}[leftmargin=16pt]
\item[1)] Any P-stationary point is binary. 
\item[2)] A local minimizer is a P-stationary point if $f$ is $L$-strongly smooth on $\B$. 
\item[3)] A  P-stationary point with $\tau\geq1/\ell$ is a global minimizer if $f$ is $\ell$-strongly convex on $\B$. 
\end{itemize} 
\end{theorem}
\proof  1) Suppose $\tilde{\x} \notin \{0,1\}^n$ is a P-stationary point of (\ref{pi-penalty}). Then, there exists a $\tau>0$ such that
\begin{equation}
\tilde{\x} \in {\rm Prox}_{\tau{\lambda} {p}}^\B\Big(\tilde{\x} - \tau \nabla f(\tilde{\x})\Big),
\end{equation} 
According to the generalized Fermat's rule \cite[Theorem 10.1(P422)]{rock98}, it follows that 
\begin{equation}\label{opt-P-sta}
	\begin{aligned}
		0&\in \partial\left(\frac{1}{2}\|\tilde{\x} -(\tilde{\x}- \tau \nabla f(\tilde{\x}))\|^2 + \tau{\lambda}  {p}(\tilde{\x}) + \delta_\B(\overline{\x})\right)\\
		&=  \tau \nabla f(\tilde{\x}) + \partial\left( \tau{\lambda}  {p}(\tilde{\x}) + \delta_\B(\tilde{\x})\right)\\[1.5ex]
		&=    \tau \nabla f(\tilde{\x})  +  \tau{\lambda}  \partial  {p}(\tilde{\x}) + N_\B(\tilde{\x}),
\end{aligned}
\end{equation} 
where the first equation are from three facts: 1) $f$ is continuously differentiable, 2) ${p}(\widetilde{\x})$ is lower semi-continuous and bounded,  and  3) \cite[10.10 Exercise]{rock98}, the second equation holds because the unique nondifferentiable point of $g$ is $1/2$, which lies in the interior of $[0,1]$.

Since $\tilde{\x} \notin \{0,1\}^n$, there is an ${\widetilde{x}_i \in (0,1)}$. Then \eqref{opt-P-sta} implies that there is ${\nu_i\in \partial g(\tilde{x}_i)}$ such that $\nabla_i f (\widetilde{\x})= -{\lambda}\nu_i$.  Moreover, from \eqref{lowerbound}, we have $|\nu_i|> 3$ for any  ${\nu_i\in \partial g(\tilde{x}_i)}$, which yields
\begin{equation} \label{opt-contradiction}
	\overline{{\lambda}}  =   \max_{\x\in\B}\frac{\|\nabla f(\x)\|_\infty}{3}  \geq   \frac{|\nabla_i f(\widetilde{\x})|}{3}   =       \frac{\lambda|\nu_i|}{3}  >  \lambda \geq  \overline{{\lambda}},
\end{equation}
  leading to a contradiction. Therefore, any P-stationary point must be binary.

2) Let $\x$ be a local minimizer of \eqref{pi-penalty} and $\z$ satisfy $\z \in {\rm Prox}_{\tau{\lambda} {p}}^\B\Big(\x - \tau \nabla f(\x)\Big)$. By Fermat’s rule \cite[Theorem 10.1(P422)]{rock98}, it holds that
\begin{eqnarray} \label{opt-local}
	\begin{aligned}
		0 &\in \partial (f(\widetilde{\x})+{\lambda} {p}(\widetilde{\x}) + \delta_\B(\widetilde{\x}))\\
		&= \nabla f(\widetilde{\x})+{\lambda} \partial  {p}(\widetilde{\x}) + N_\B(\widetilde{\x}),
	\end{aligned}
\end{eqnarray}
where the equality holds for the same reason as in the previous part. Similar to the argument in \eqref{opt-contradiction}, it follows from \eqref{opt-local} that $\x$ is binary, and hence $p(\x)=0$. The definition of ${\rm Prox}_{\tau{\lambda} {p}}^\B$ yields
\begin{equation*} 
	\begin{aligned}
		\| \z -({\x}- \tau \nabla f({\x}))\|^2+2\tau{\lambda} {p}(\z) \leq  \| \x -({\x}- \tau \nabla f({\x}))\|^2+2\tau{\lambda} {p}(\x), 
\end{aligned}\end{equation*} 
which results in 
\begin{equation} \label{cond-P-sta}
\begin{aligned}
 \| \z - {\x} \|^2 &\leq 2\tau\langle {\x}-\z , \nabla f({\x}) \rangle  +  2\tau{\lambda} ({p}(\x)-{p}(\z))\\
 &\leq 2\tau\langle {\x}-\z , \nabla f({\x}) \rangle   \\
 &\leq 2\tau\|\z- {\x}\|\| \nabla f({\x})\|   \\
  &\leq 2\tau\sqrt{n}\|\z- {\x}\|\| \nabla f({\x})\|_{\infty}   \\
   &\leq 6\tau \overline{{\lambda}}\sqrt{n}\|\z- {\x}\|,
\end{aligned}\end{equation} 
where the last inequality is from \eqref{lower-bd-pi}. This implies $\| \z - {\x} \|\leq 2\tau c\overline{{\lambda}}\sqrt{n}$, meaning $\z$ is around ${\x}$ when $\tau$ is sufficiently small. 
Using the $L$-strong smoothness of $f$ on $\B$ and the first inequality in \eqref{cond-P-sta}, we obtain
 \begin{equation*} 
\begin{aligned}
&~2F(\z;{\lambda})- 2F(\x;{\lambda})\\
\leq &~2\langle \z - {\x}, \nabla f({\x}) \rangle  + 2{\lambda} ({p}(\z)-{p}(\x))+L\| \z - {\x} \|^2\\
\leq&~\left({L}-{1}/{\tau }\right)\| \z - {\x} \|^2 \leq 0,
\end{aligned}\end{equation*}
where the last inequality is because $\tau$ is sufficiently small.
The above condition leads to  $\z = {\x} $ due to the local optimality of $\x$, namely, $\x$ is a P-stationary point.

3) Let $\overline{\x}$ be a P-stationary point.  Then it satisfies
\begin{equation*} 
\begin{aligned}
 \overline{\x} \in {\rm Prox}_{\tau{\lambda} {p}}^\B\Big(\overline{\x} - \tau \nabla f(\overline{\x})\Big).
\end{aligned}\end{equation*} 
By the definition of ${\rm Prox}_{\tau{\lambda} {p}}^\B$, we obtain
\begin{equation*} 
\begin{aligned}
 \| \overline{\x} -(\overline{\x}- \tau \nabla f(\overline{\x}))\|^2+2\tau{\lambda} {p}(\overline{\x}) \leq  \| \w -(\overline{\x}- \tau \nabla f(\overline{\x}))\|^2+ 2\tau{\lambda} {p}(\w), 
\end{aligned}\end{equation*} 
for any $\w\in\B$,  which results in 
\begin{equation}  
\begin{aligned}
2\langle \overline{\x} - {\w}, \nabla f(\overline{\x}) \rangle  + 2{\lambda} ({p}(\overline{\x})-{p}(\w))  &\leq  ( {1}/{\tau })\| \w - \overline{\x} \|^2. 
\end{aligned}\end{equation} 
Using the above condition and the $\ell$-strong convexity of $f$, we have
 \begin{equation*} 
\begin{aligned}
&~2F(\overline{\x};{\lambda}) - 2F(\w;{\lambda})\\
\leq  &~2\langle \overline{\x} - {\w}, \nabla f(\overline{\x}) \rangle  + 2{\lambda} ({p}(\x)-{p}(\w))- \ell\| \w - \overline{\x} \|^2\\
\leq &~\left({1}/{\tau }-{\ell}\right)\| \z - \overline{\x} \|^2 \leq 0,
\end{aligned}\end{equation*}
due to $\tau\geq1/\ell$. Therefore, $\overline{\x}$ is global minimizer of \eqref{pi-penalty}. 
\qed

To highlight the advantages of our proposed penalty function $g(x)$ over existing alternatives, we provide a comparative example. Specifically, we contrast our piecewise cubic penalty with the $\ell_p$-norm based penalty $h(x)=n-\|2\mathbf{x}-\mathbf{1}\|_p^p$ from \cite{Wu19}.

\begin{example}
	\label{ex:comparison}
	Consider the following two single-variable penalty problems:
	\begin{eqnarray}
		\label{ex-pi-penalty-1}&&\min_{x\in [0,1]}~ F_1(x;{\lambda})=\frac{1}{2}\left(x- \frac{1}{2}\right)^2+{\lambda}(1-|2x-1|^3),\\
		\label{ex-pi-penalty-2}&&\min_{x\in [0,1]}~ F_2(x;{\lambda})=\frac{1}{2}\left(x- \frac{1}{2}\right)^2+{\lambda} g(x).
	\end{eqnarray} 
	
	\begin{itemize}[leftmargin=15pt]
		\item \textbf{Problem (\ref{ex-pi-penalty-1}) with $\ell_p$-norm penalty:} For any ${\lambda}>0$, the non-binary point $x=1/2$ is always a local minimizer. To verify this, note that $\partial F_1({1}/{2};{\lambda})=\{0\}$ for all ${\lambda}$. Furthermore, for sufficiently small $|\varepsilon|$:
		\begin{eqnarray*}
			\begin{aligned} 
				F_1\left(\frac{1}{2}+\varepsilon;{\lambda}\right)-F_1\left(\frac{1}{2};{\lambda}\right)&= \frac{1}{2}\varepsilon^2+{\lambda} \left(1-8|\varepsilon|^3\right) - \lambda\\
				&=  \left(\frac{1}{2|\varepsilon|}-8{\lambda}\right) |\varepsilon|^3 \geq 0 ,
		\end{aligned}\end{eqnarray*}
		confirming that $x=1/2$ is indeed a local minimizer regardless of the penalty parameter value.
		
		\item \textbf{Problem (\ref{ex-pi-penalty-2}) with our piecewise cubic penalty:} When ${\lambda}>\overline{\lambda}=1/6$, all local minimizers are binary. To see this, observe that
		\begin{equation*}
			0\notin  \partial F_2(x;{\lambda}) =  x- \frac{1}{2} +{\lambda}\begin{cases}
				\{3x^2-6x+3\},& x \in(0,1/2),\\
				[-3/4,3/4],& x =1/2,\\
				\{-3x^2\},& x \in(1/2,1).
			\end{cases}
		\end{equation*} 
		Hence, Problem (\ref{ex-pi-penalty-2}) has no local minimizers in $(0,1)$. The binary points $x=0$ and $x=1$ are the only local (and global) minimizers.
	\end{itemize}
\end{example}

This example reveals a key distinction: penalties like $h(x)$ may admit non-binary local minimizers regardless of the penalty parameter, whereas  $g(x)$ ensures all local minimizers are binary when ${\lambda}$ exceeds a computable threshold.

Based on Theorem \ref{firstPST}, we establish the exact penalty theorem for problems \eqref{SCO} and \eqref{pi-penalty}. 
\begin{theorem}[Exact Penalty Theorem]\label{theorem-ept} A point is a global minimizer of  problem (\ref{SCO})  if and only if it is a global minimizer of problem (\ref{pi-penalty}).
\end{theorem}
\proof  

Let $\x^*$ and $\widetilde{\x}$ be global minimizer of problems \eqref{SCO} and \eqref{pi-penalty}. Similar to the reasoning in the proof of part 3) of Theorem \ref{firstPST}, any global minimizer of \eqref{pi-penalty} satisfies the optimality condition \eqref{opt-local} and thus $\widetilde{\x}\in \{0,1\}^n$.   Note that the feasible region of \eqref{SCO} is $ \{0,1\}^n$. Since $\widetilde{\x}\in \{0,1\}^n$, we have $f(\x^*)\leq f(\widetilde{\x})$ by the global optimality of $\x^*$ for problem \eqref{SCO}. This yields   
\begin{equation*}
	F(\widetilde{\x};{\lambda})  \leq  F({\x}^*;{\lambda})  
	=
	f(\x^*)\leq f(\widetilde{\x}) 
	=   F(\widetilde{\x};{\lambda}),
\end{equation*} 
where the first inequality holds because $\widetilde{\x}$ is a global minimizer of problem \eqref{pi-penalty} and ${\x}^*\in\{0,1\}^n \subseteq\B$, the two equations hold due to ${p}({\x}^*)={p}(\widetilde{\x})=0$. The above condition implies the four values are the same, concluding the conclusion.  \qed

\begin{remark}
	We make a few comments on Theorem \ref{firstPST} and \ref{theorem-ept}.

Based on Theorems \ref{firstPST}$-$\ref{theorem-ept}, we summarize the relationships
		among different solutions to problems (\ref{UBIP}), (\ref{SCO}), and (\ref{pi-penalty}), as shown in Figure \ref{fig:relation}. Specifically, the global minimizers of these three problems are equivalent;  a global minimizer $\x$ of problem (\ref{pi-penalty})  is a P-stationary point if  $f$ is strongly smooth on $\B$; a P-stationary point of problem (\ref{pi-penalty}) is also a global minimizer if $f$ is strongly convex on $\B$. These relationships indicate that finding a P-stationary point of problem \eqref{pi-penalty} provides an effective approach for solving problem \eqref{SCO} or \eqref{UBIP}.

		 Theorem \ref{theorem-ept} shows that when the penalty parameter $\lambda>\overline{\lambda}$, the global minimizers of the penalty model (\ref{pi-penalty}) and the original problem (\ref{SCO}) coincide. Recall that $\overline{\lambda}$ defined in (\ref{lower-bd-pi}) is a constant independent of the solution set of (\ref{SCO}), which distinguishes Theorem \ref{theorem-ept} from traditional exact penalty theory (e.g., \cite[Theorem 17.3]{Nocedal99}), where the penalty parameter threshold typically depends on the solution to the original problem. Additionally, a result similar to Theorem \ref{theorem-ept} was established in \cite{Yuan17}, but under the assumption that $f$ is Lipschitz continuous and convex on $\B$. However, deriving Theorem \ref{theorem-ept} requires no additional conditions on $f$.

\end{remark}

 \begin{figure}[!t]
 \centering
 \includegraphics[scale=.55]{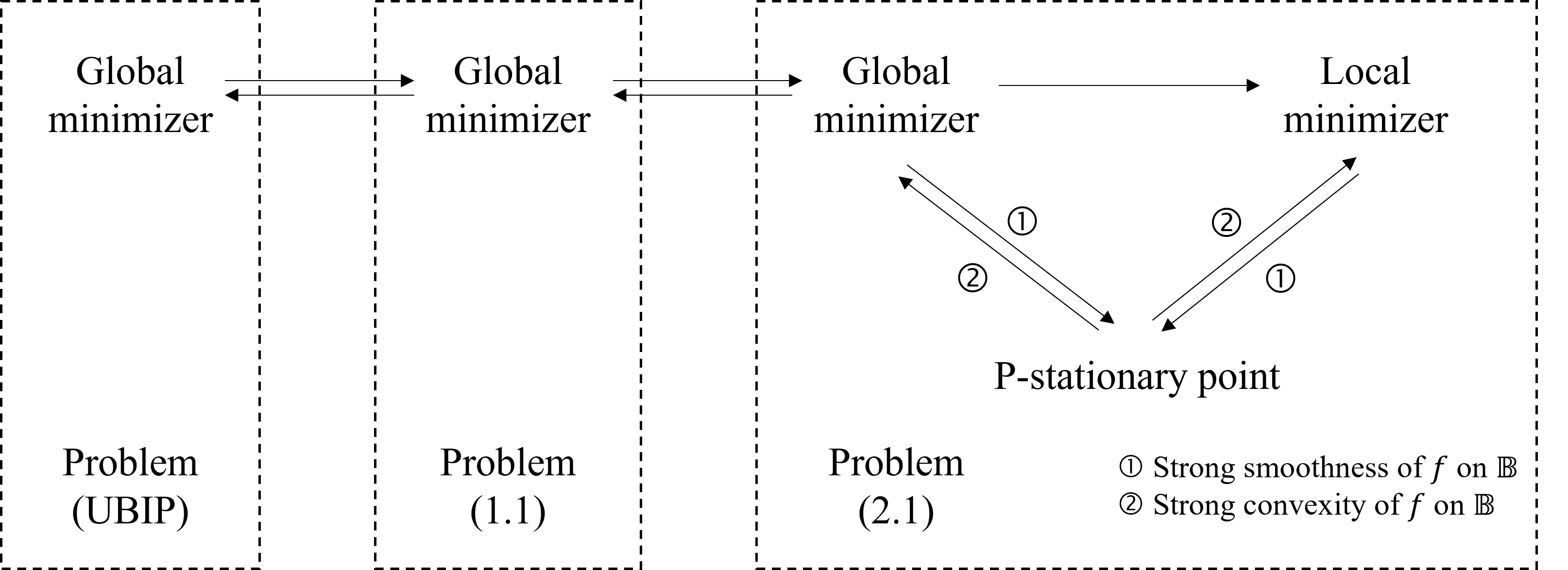}
 \caption{Relationships among different points for problems \eqref{UBIP} and \eqref{pi-penalty}.\label{fig:relation}}
 \end{figure}

To end this section, we present a result to guarantee a binary solution to  ${\rm Prox}_{ \tau{\lambda} {p}}^{\B}$. 
\begin{lemma}\label{P-binary} Let $\x\in\B$ and $\z\in {\rm Prox}_{ \tau{\lambda} {p}}^{\B}(\x - \tau \nabla f(\x ))$. Then $\z\in\{0,1\}^n$ when ${\lambda}\geq\overline{{\lambda}} + 1/(3\tau)$.
\end{lemma}
\begin{proof} By the definition of ${\rm Prox}_{ \tau{\lambda} {p}}^{\B}$, we have 
\begin{equation*} 
\begin{aligned}
\z \in  {\rm argmin}_{\w\in\B}~\frac{1}{2}\| \w -({\x}- \tau \nabla f({\x}))\|^2+\tau{\lambda} {p}(\w).
\end{aligned}\end{equation*} 
Then the similar reasoning to derive   \eqref{opt-P-sta} enables the following optimality condition,
$$0\in\z - {\x} + \tau \nabla f({\x}) +\tau{\lambda} \partial\varphi(\z)+N_{\B}(\z).$$
If there is $i$ such that $z_i\in(0,1)$, then $z_i- x_i + \tau \nabla_i f({\x}) +\tau{\lambda} v_i=0$, where $\v\in\partial\varphi(\z)$. Hence,
\begin{equation*} 
\begin{aligned}
3\tau{\lambda} &< | \tau{\lambda}v_i|=| z_i- x_i + \tau \nabla_i f({\x})|\\
&\leq | z_i- x_i| + \tau |\nabla_i f({\x})|\\
&\leq 1 + 3\tau \overline{{\lambda}} \\
&< 3\tau{\lambda} ,
\end{aligned}\end{equation*} 
where the third inequality is from $z_i\in(0,1), x_i\in[0,1]$, and \eqref{lower-bd-pi}. This contradiction implies that $\z$ is binary, as desired.
\end{proof}
%

%


 \section{Algorithm and convergence} \label{Section-algorithm}

In this section, we develop the algorithm to solve problem \eqref{pi-penalty}. Specifically, for current point $\x^k\in\R^n$, the next point is updated by,
\begin{eqnarray}\label{Inter9121}\x^{k+1} \in {\rm Prox}_{\tau_k{\lambda}_k {p}}^{\B}\Big(\x^k - \tau_k \nabla f(\x^k)\Big),\end{eqnarray}
where $\tau_k$ and ${\lambda}_k$ are updated adaptively.  Stopping criteria can be set as
\begin{eqnarray}\label{stop}\x^{k}\in\{0,1\}^n, \qquad\left\|\x^{k} -  \x^{k+1}\right\|=\left\|\x^{k} - {\rm Prox}_{\tau_k{\lambda}_k {p}}^{\B}\left(\x^k - \tau_k  \nabla f(\x^k)\right)\right\|<\varepsilon,\end{eqnarray} 
where $\varepsilon\in(0,1)$ is a given tolerance. One can observe that if $\x^k$ satisfies the above conditions for any given $\varepsilon$, then it is a binary P-stationary point of problem \eqref{pi-penalty} with a particular ${\lambda}$, as outlined in Theorem \ref{global-convergence}. The main update, \eqref{Inter9121}, is a typical step of  the proximal point algorithm. To facilitate binary solutions, we increase penalty parameter $\lambda$ adaptively. Overall, we present the algorithmic framework in Algorithm \ref{BN-algo} and call it an adaptive proximal point algorithm (APPA).

\begin{algorithm}[!th]
    \SetAlgoLined

 \textbf{Input} $\x^0\in \B$,  $({\lambda}_0, \eta, \sigma, \theta)>0, \alpha\in(0,1)$, $\pi>1$,  and an integer $k_0>0$.
 
	\For{$k=0,1,2,\ldots$}{
	
	 Find the smallest integer $s_k\in\{0,1,2,\ldots\}$ such that 
	\begin{align}\label{eta-x-F-k-1}
	\tau_k &~=~ \eta \alpha^{s_k},\\
	\label{eta-x-F-k-2}	\x^{k+1} &~\in~ {\rm Prox}_{\tau_k{\lambda}_k {p}}^{\B}\Big(\x^k - \tau_k \nabla f(\x^k)\Big),\\
	\label{eta-x-F-k-3}	F(\x^{k+1};{\lambda}_{k})&~\leq~ F(\x^{k};{\lambda}_{k}) - {(\sigma/2)}\|\x^{k+1}-\x^{k}\|^2.
	\end{align} 
	
	 If \eqref{stop} holds, then stop.
	   
	 If mod$(k+1,k_0)=0$ and ${\lambda}_{k}<\theta$, then ${\lambda}_{k+1}=\pi {\lambda}_{k},$ else ${\lambda}_{k+1}= {\lambda}_{k}.$
}
 
	\caption{Adaptive proximal point algorithm (APPA).}\label{BN-algo}
\end{algorithm}

In Algorithm \ref{BN-algo}, we increase the penalty parameter only when mod$(k+1,k_0)=0$ and ${\lambda}_{k}<\theta$, where mod$(a,b)$ returns  the remainder after division of $a$ by $b$. This is because condition mod$(k+1,k_0)=0$ prevents ${\lambda}_{k}$ from increasing too quickly, while threshold $\theta$ ensures that ${\lambda}_{k}$ does not tend to $\infty$. According to Lemma \ref{P-binary}, there exists a finite threshold beyond which   $\x^{k+1}$ is ensured to be binary. Therefore, there is no reason to continue increasing ${\lambda}_{k}$ once it exceeds the threshold.

\subsection{Convergence analysis}
To establish the convergence, we always set 
\begin{equation}\label{k-lower-bd}
\theta \geq  \overline{{\lambda}}+\frac{  \sigma+L }{3\alpha}.
\end{equation} 
Based on this $\theta$, we define some constants:
\begin{equation}\label{def-constants}
\begin{aligned}
 \tau_\infty &:=\eta \alpha^{S} \qquad &&\text{with}\qquad S:=-\left\lceil\log_{\alpha}(\eta(\sigma+L)\right\rceil,\\
{\lambda}_\infty&:= {\lambda}_0\pi^K \qquad &&\text{with}\qquad K:=\left\lceil k_0 \max\left\{0,\log_{\pi} \left(\frac{  \theta}{ {\lambda}_0}    \right)\right\}\right\rceil.
\end{aligned}
\end{equation}
Based on the above constants, our first result shows that $\tau_k$ is well defined,  the point generated by APPA is always binary, and sequence $\{f(\x^k)\}$ is non-increasing after finitely many iterations. 
\begin{lemma}\label{lemma-decreasing-f} The following claims hold for  APPA  if $f$ is $L$-strong smoothness on $\B$.
\begin{itemize}[leftmargin=16pt]
\item[1)] For any $k\geq 1$, conditions \eqref{eta-x-F-k-1}-\eqref{eta-x-F-k-3} are ensured with $\tau_k = \tau_\infty$ and
\begin{equation}\label{eta-k-bd}
\frac{1}{\sigma+L} \geq \tau_\infty \geq \frac{\alpha}{\sigma+L}.
\end{equation} 
\item[2)]  For any $k> K$, it holds ${\lambda}_k \equiv {\lambda}_\infty$, $\x^{k}\in\{0,1\}^n$,  and 
\begin{equation}\label{decreasing-property}
\begin{aligned}
f(\x^{k+1})  -f(\x^{k})  \leq   \frac{\sigma}{2}  \|\x^{k+1}-\x^k \|^2.
\end{aligned}
\end{equation} 
\end{itemize}
\end{lemma}

\proof 1) For any $\tau>0$, the following problem
	\begin{eqnarray*}
		 \w^{k+1} \in {\rm Prox}_{\tau{\lambda}_k {p}}^{\B}\left(\x^k - \tau  \nabla f(\x^k)\right),
	\end{eqnarray*} 
	enables us to derive that 
\begin{equation*}
\begin{aligned}
 \|\w^{k+1}-\x^k + \tau  \nabla f(\x^k)\|^2+ 2\tau{\lambda}_{k} {p}(\w^{k+1})
 &\leq  \|\x^{k}-\x^k + \tau  \nabla f(\x^k)\|^2+ 2\tau{\lambda}_{k} {p}(\x^k),
\end{aligned}
\end{equation*}
which leads to
\begin{equation*}
2{\lambda}_{k}  {p}(\w^{k+1})-  2{\lambda}_{k}{p}(\x^{k})  +2\langle \nabla f(\x^k), \w^{k+1}-\x^k\rangle \leq -({1}/{ \tau})\|\w^{k+1}-\x^k\|^2.
\end{equation*}
Using the above condition and the $L$-strongly smoothness of $f$ yield that
\begin{equation*}
\begin{aligned}
&2F(\w^{k+1};{\lambda}_{k})-2F(\x^{k};{\lambda}_{k})\\  
\leq ~& 2{\lambda}_{k} {p}(\w^{k+1})-  2{\lambda}_{k}{p}(\x^{k})  + 2\langle \nabla f(\x^{k}), \w^{k+1}-\x^{k}\rangle +L\| \w^{k+1}-\x^{k}\|^{2} \\
\leq ~& \left( {L} - {1}/\tau\right)\|\w^{k+1}-\x^k\|^2.
\end{aligned}
\end{equation*}
Therefore, if $\tau\in(0,1/(\sigma+L)]$, then we have
$$ F(\w^{k+1};{\lambda}_{k})-F(\x^{k};{\lambda}_{k})  \leq - \frac{\sigma}{2} \|\w^{k+1}-\x^k\|^2.$$
This means conditions \eqref{eta-x-F-k-1}-\eqref{eta-x-F-k-3} hold with $\w^{k+1}=\x^{k+1}$ when $$\tau_k= \eta \alpha^{s_k}=  \eta \alpha^{S} \in\left[\frac{\alpha}{\sigma+L},\frac{1}{\sigma+L}\right].$$
2) If ${\lambda}_0\geq \theta$, then we always have ${\lambda}_{k}\equiv{\lambda}_0$ as $\{{\lambda}_{k}\}$ is non-decreasing. Therefore,
 \begin{equation} \label{pi-eta-great-half}
  {\lambda}_{k}\tau_{k} - \overline{{\lambda}} \tau_{k}\geq \frac{ ({\lambda}_0 - \overline{{\lambda}})\alpha}{\sigma+L} \geq \frac{ (\theta - \overline{{\lambda}})\alpha}{\sigma+L} \geq \frac{1}{3},
 \end{equation}
 where the first and last inequalities are from \eqref{eta-k-bd}  and \eqref{k-lower-bd}. Thus  $\x^{k+1}\in\{0,1\}^n$ from Lemma \ref{P-binary}. We note that if ${\lambda}_0\geq \theta$ then $K=0$, thereby leading to the conclusion.
 
 Now we consider ${\lambda}_0<\theta$. It follows from \eqref{def-constants} that
  \begin{equation*} 
 {\lambda}_{K} =  {\lambda}_0\pi^{\lceil K/k_0\rceil} \geq \theta. 
 \end{equation*}
Therefore, for any $k\geq K$, we have ${\lambda}_{k}\equiv{\lambda}_{K}\equiv{\lambda}_\infty$. Then  similar reasoning to show \eqref{pi-eta-great-half} can show $ {\lambda}_{k}\tau_{k}\geq  \overline{{\lambda}} \tau_{k}+1/3$, which by Lemma \ref{P-binary} yields $\x^{k}\in\{0,1\}^n$  and thus ${p}(\x^k)=0$. Overall
 \begin{equation*} 
 {\lambda}_k\left\{
 \begin{array}{ll}
 \in\left[{\lambda}_0, {\lambda}_\infty\right),&~ k\in[0,K],\\[1ex]
 \equiv{\lambda}_\infty,&~ k> K,
 \end{array} \right.
 \end{equation*}
 namely, $\{{\lambda}_k\}$ is bounded. Based on these facts,  for any $k> K$, we have
\begin{equation*}
\begin{aligned}
f(\x^{k+1})  -f(\x^{k}) &=
 F(\x^{k+1};{\lambda}_{k+1})-F(\x^{k};{\lambda}_{k})\\
 &=
 F(\x^{k+1};{\lambda}_{\infty})-F(\x^{k};{\lambda}_{\infty})\\
&\leq   ({\sigma}/{2})  \|\x^{k+1}-\x^k \|^2.
\end{aligned}
\end{equation*}
where the last inequality is from \eqref{eta-x-F-k-3}. 
  \qed
\begin{theorem} \label{global-convergence}
The following results hold for  APPA if $f$ is $L$-strong smoothness on $\B$.
\begin{itemize}[leftmargin=16pt]
\item[1)] Whole sequence $\{f(\x^k)\}$ converges and $\lim_{k\to\infty}\|\x^{k+1}-\x^k \|=0$. 
\item[2)] Whole sequence $\{\x^k\}$ converges to a P-stationary point with $\tau\geq 1/(3{\lambda}_\infty-3\overline{{\lambda}})$ of 
\begin{equation}
\label{pi-infty-penalty}
\min_{\x\in \B}  f(\x)+{\lambda}_\infty {p}(\x).
\end{equation}
\end{itemize}
\end{theorem}
\proof 1) Since $\{\x^k\}\subseteq\B$,  we have $\{f(\x^k)\}$ is bounded because $f$ is continuous. Therefore, $\{f(\x^k)\}$ converges by \eqref{decreasing-property}. Taking the limit of the both sides of \eqref{decreasing-property} derives $\lim_{k\to\infty}\|\x^{k+1}-\x^k \|=0$. 

2)  By \eqref{pi-eta-great-half}, for any $k> K$ we have \begin{equation} \label{pi-infty-eta-infty}
 {\lambda}_\infty\tau_\infty =  {\lambda}_{k+1}\tau_{k+1} \geq   \overline{{\lambda}} \tau_\infty + {1}/{3}.
 \end{equation}
 Let $\{ \x^k: k\in T\}$ be a convergent subsequence  of $\{\x^k\}\subseteq\B$ and $\lim_{k(\in T)\to\infty}   \x^k  =   \x^\infty,$ where $T\subseteq\{0,1,2,3\ldots\}$. 
Moreover, it follow from Lemma \ref{lemma-decreasing-f} that 	
\begin{align*}
\x^{k+1}  &\in  {\rm Prox}_{\tau_\infty{\lambda}_\infty {p}}^{\B}\Big(\x^k - \tau_\infty \nabla f(\x^k)\Big)\\
&= {\rm Prox}_{\tau_\infty{\lambda}_\infty ({p}+\delta_\B)} \Big(\x^k - \tau_\infty \nabla f(\x^k)\Big),
\end{align*}
for any $k> K$. Using these facts, \cite[Theorem 1.25]{rock98}, and taking the limit of both sides of the above condition along with $k(\in T)\to\infty$ enable us to derive that
\begin{align*} \x^\infty & \in   {\rm Prox}_{\tau_\infty{\lambda}_\infty  ({p}+\delta_\B)} \Big( \x^\infty- \tau_\infty \nabla f( \x^\infty)\Big)\\
&={\rm Prox}_{\tau_\infty{\lambda}_\infty   {p} }^\B \Big( \x^\infty- \tau_\infty \nabla f( \x^\infty)\Big).\end{align*}
Therefore $\x^\infty$ is a P-stationary point of \eqref{pi-infty-penalty}  with $\tau_\infty\geq 1/(3{\lambda}_\infty-3\overline{{\lambda}})$. Since any $\x^\infty\in\{0,1\}^n$ which means it is isolated around its neighbour region. This together with $\lim_{k\to\infty}\|\x^{k+1}-\x^k \|=0$ and \cite[Lemma 4.10]{more83} delivers the  whole sequence convergence. \qed


\begin{theorem}\label{finite-termination}
  APPA terminates within finitely many steps if $f$ is $L$-strong smoothness on $\B$. 
\end{theorem}
\proof   
For any $k> K$, if $\x^{k+1}\neq\x^k$, we have $ \|\x^{k+1}-\x^k \|>1$ due to  $\x^{k}\in\{0,1\}^n$,  contradicting to $\lim_{k\to\infty}\|\x^{k+1}-\x^k \|=0$.  Therefore, we can conclude that there is $K'> K$ such that $\x^{k+1}\equiv\x^k$ for any $k\geq K'$, which indicates conditions in \eqref{stop} are satisfied. \qed

\section{Numerical experiment}
In this section, we evaluate the performance of APPA on three problems: the recovery problems, the Multiple-Input-Multiple-Output (MIMO) detection, and the quadratic unconstrained binary optimization (QUBO). All codes are implemented in MATLAB (R2021a) and executed on a desktop computer with an Intel(R) Xeon W-3465X CPU (2.50 GHz) and 128 GB of RAM.

\subsection{Recovery problem}\label{sec:rec}
We first demonstrate the performance of APPA for solving a recovery problem. The problem aims to recover a binary signal $\x^*$ from the linear system ${{\textbf b}={\textbf A}\x^*}$ with ${{\textbf A}\in\R^{m\times n}}$ and ${\textbf b}\in \R^m$. To find the original signal $\x^*$, we consider the following optimization problem,
 \begin{eqnarray*}
\min\limits_{\x\in{\R^{n}}}~~ \frac{1}{2}\|\textbf{A}\x-\textbf{b}\|_q^{q},~~~
{\rm s.t.}~~   \x\in \{0,1\}^{n},
\end{eqnarray*}
where ${\|\x \|_q^{q}=\sum_i|x_i|^q}$ and ${q>1}$. 
Let ${\textbf b}\in \R^m$ be generated by the linear regression model
$${{\textbf b}={\textbf A}\x^*+ {\rm nf}\cdot {\boldsymbol \varepsilon}},$$
where ${\textbf A}\in\R^{m\times n}$ is the measurement matrix whose entries are independently and identically distributed (i.i.d.) from the standard normal distribution ${\cal N}(0,1)$, followed by normalization ${\textbf A}\leftarrow{\textbf A}/c$ with ${c=\sqrt{m}}$ when ${n\le 10000}$ and ${c=1}$ otherwise. The ground truth signal $\x^*$ is taken from $\{0,1\}^n$ with randomly picked $s$ indices on which entries are $1$. The additive noise ${\boldsymbol \varepsilon}\in\R^m$ has i.i.d. components ${\varepsilon_i\sim\mathcal{N}(0,1)}$ for $i=1,\ldots,m$, and the parameter ${\rm nf}\geq 0$ controls the noise level.

\subsubsection{Benchmark methods}

We compare APPA against NPGM, MEPM \cite{Yuan17}, L2ADMM \cite{Wu19}, EMADM \cite{liu23}, and GUROBI, where NPGM applies Nesterov's proximal gradient method to the LP relaxation of \eqref{UBIP} and EMADM solves its SDP relaxation. Note that GUROBI and EMADM only handle binary quadratic programs ($q = 2$). For APPA, we set $(\eta,\sigma,\lambda,\pi)=(1,10^{-8},0.25,1.5)$, ${\theta=\|\textbf{A}\|_{\infty}+\|\textbf{b}\|_{\infty}}$, update frequency $k_0=100$ (if $n<10000$) or $k_0=50$ (otherwise), and initial penalty $\lambda_0=r\|\textbf{b}^{\top}\textbf{A}\|_{\infty}$ with $r=\min\{0.01,0.1^{4\sqrt{s}/\log_2(mn)}\}$ if $2m\leq n < 10s$ and $r=0.05$ otherwise. For MEPM and L2ADMM, we set $(\rho, \sigma, T)=(0.01,\sqrt{10},10)$ and $(\alpha, \sigma, T)=(0.01,\sqrt{10},10)$ respectively, both with 200 maximum iterations, tolerance $10^{-2}$, and Lipschitz constant $L = \|\textbf{A}\|^2_F/n$. For EMADM, we set $(\beta, \lambda_0, \delta)=(0.1,0.001,10^{-4})$ with $10^4$ maximum iterations and tolerance $10^{-4}$. Finally, we initialize all algorithms by $\x_0 = 0$ and set the maximum running time to 600 seconds. To evaluate the algorithmic performance, we let $\x$ be the solution obtained by one algorithm and report the recovery accuracy (${\rm Acc}:= 1- \| \x-\x^* \|/\| \x^*\|$) and the computational time in seconds. 

\subsubsection{Numerical comparison}

To evaluate the efficiency of selected algorithms for solving the recovery problem in different scenarios. We alter one factor of $(m,n,s,q,\rm{nf})$ to see its effect by fixing the other four factors.

\begin{figure}[!ht]
	\centering
	
	\begin{subfigure}[b]{0.48\textwidth}
		\centering
		\includegraphics[width=0.98\textwidth]{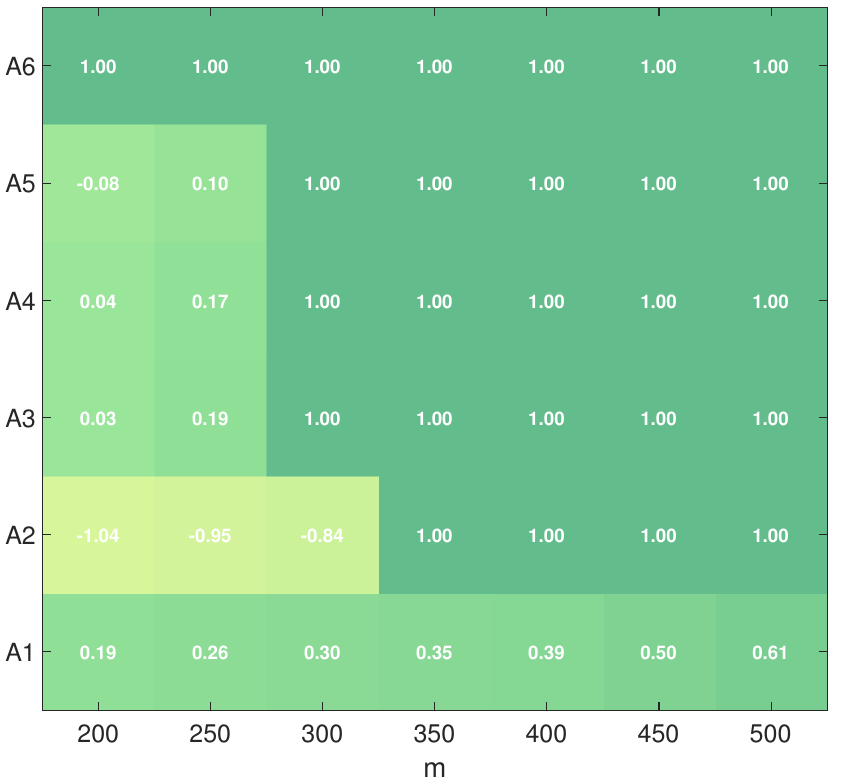}
		\caption{Acc v.s. $m$}
		\label{fig:effect-m}
	\end{subfigure}
	\begin{subfigure}[b]{0.48\textwidth}
		\centering
		\includegraphics[width=0.98\textwidth]{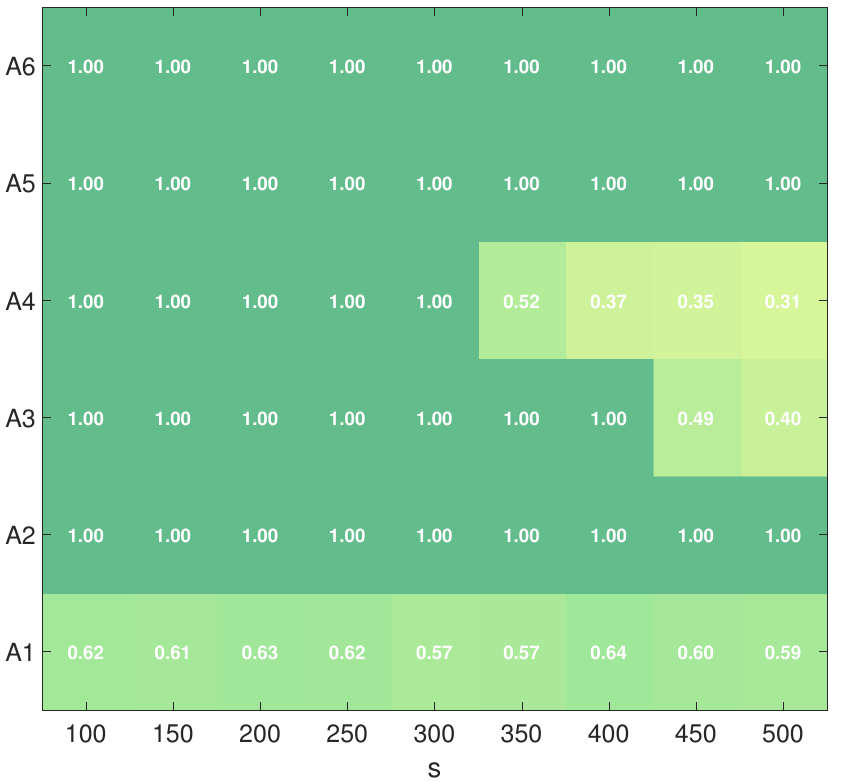}
		\caption{Acc v.s. $s$}
		\label{fig:effect-s}
	\end{subfigure}
	
	\vspace{0.2cm}
	
	\begin{subfigure}[b]{0.48\textwidth}
		\centering
		\includegraphics[width=0.98\textwidth]{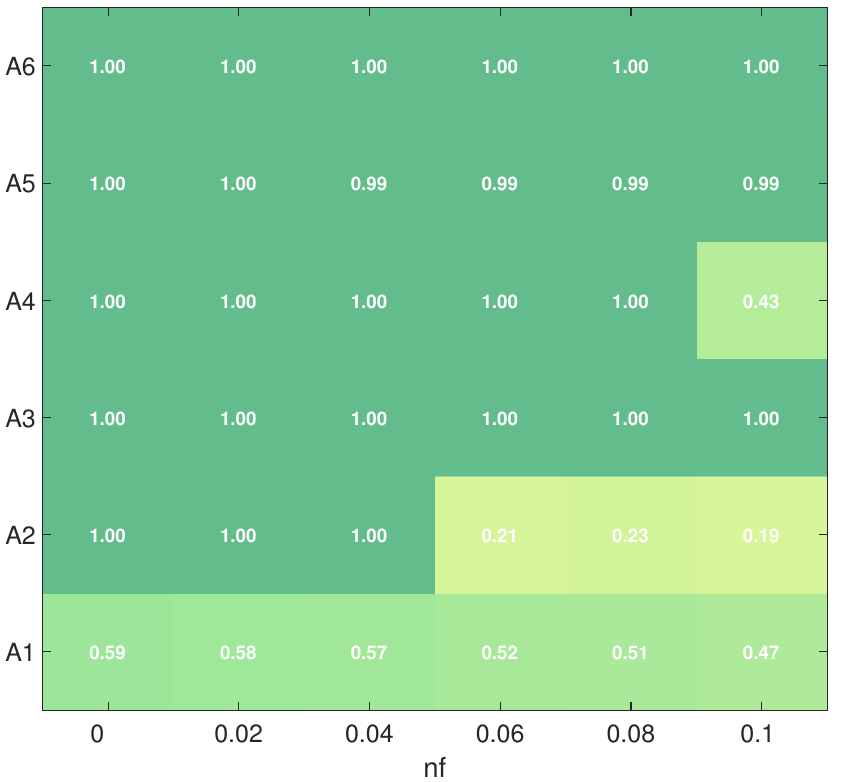}
		\caption{Acc v.s. $\rm{nf}$}
		\label{fig:effect-nf}
	\end{subfigure}
		\begin{subfigure}[b]{0.48\textwidth}
		\centering
		\includegraphics[width=0.98\textwidth]{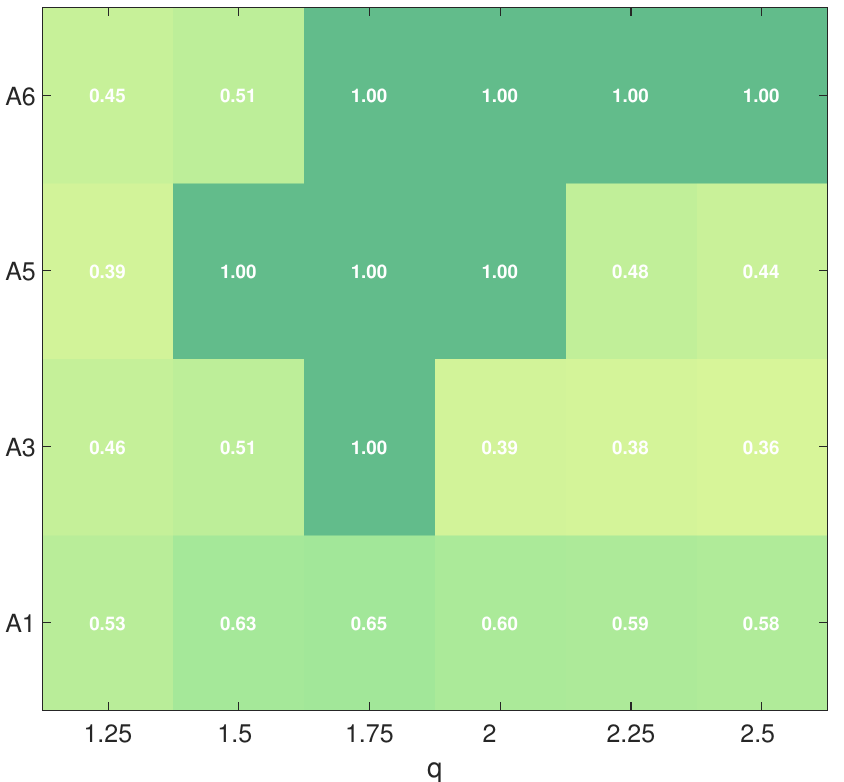} 
		\caption{Acc v.s. $q$}
		\label{fig:effect-q}
	\end{subfigure}\vspace{-2mm}
	\caption{Results on recovery problems, where A1-A6 stand for NPGM,  GUROBI,  MEPM, EMEDM, L2ADMM and APPA, respectively.}\vspace{-2mm}
	\label{fig:effect-m-s-nf-q}
\end{figure}

\begin{table}[!t]
	\renewcommand{\arraystretch}{1.00}\addtolength{\tabcolsep}{4pt}
	\centering
	\caption{Effect of higher dimensions, where A1-A4 stand for APPA,   MEPM,   L2ADMM, and NPGM, respectively.}
	\begin{tabular}{ccccccccccc}
		\hline
		\multirow{2}{*}{$n$} & \multirow{2}{*}{$\rm{nf}$} &  \multicolumn{4}{c}{Acc} && \multicolumn{4}{c}{Time} \\
		\cline{3-6} \cline{8-11}
		& & A1 & A2 & A3 & A4  &&  A1 & A2 & A3 & A4   \\
		\hline
		\multicolumn{11}{c}{$q=1.5$} \\ \hline
		$10^4$ & 0.0 & 1.000 & 1.000 & 0.989 & 0.627 && 0.192 & 24.32 & 39.14 & 3.085 \\
		$10^4$& 0.5 & 1.000 & 1.000 & 0.988 & 0.692 && 0.218 & 23.54 & 38.87 & 3.241 \\
		$10^5$ & 0.0 & 1.000 & 1.000 & 1.000 & -0.108 && 71.89 & 923.7 & 1214.6 & 37.96 \\
		$10^5$& 0.5 & 1.000 & 1.000 & 0.981 & -0.387 && 102.8 & 941.5 & 1298.4 & 39.12 \\
		$10^6$ & 0.0 & 1.000 & -3.397 & -3.142 & -1.805 && 4.427 & 3312.8 & 3600.0 & 117.8 \\
		$10^6$& 0.5 & 1.000 & -3.425 & -3.203 & -1.912 && 67.54 & 3198.6 & 3600.0 & 118.3 \\
		\hline
		\multicolumn{11}{c}{$q=2$} \\ \hline
		$10^4$ & 0.0 & 1.000 & 1.000 & 1.000 & 0.671 && 0.169 & 18.52 & 27.18 & 2.963 \\
		$10^4$& 0.5 & 1.000 & 1.000 & 0.968 & 0.639 && 0.176 & 23.69 & 34.21 & 3.105 \\
		$10^5$ & 0.0 & 1.000 & -5.594 & -5.467 & -4.586 && 2.521 & 1261.7 & 1239.8 & 36.87 \\
		$10^5$& 0.5 & 1.000  & -5.597 & -5.488 & -4.551 && 4.037 & 1272.4 & 1283.5 & 37.45 \\
		$10^6$ & 0.0 & 1.000 & -5.578 & -4.106 & -4.547 && 5.628 & 3225.9 & 3438.2 & 106.9 \\
		$10^6$& 0.5 & 1.000 & -5.541 & -4.132 & -4.573 && 6.891 & 3287.3 & 3421.6 & 109.7 \\
		\hline
		\multicolumn{11}{c}{$q=2.5$} \\ \hline
		$10^4$ & 0.0 & 1.000 & -5.851 & -5.822 & -5.718 && 0.201 & 39.24 & 38.79 & 3.146 \\
		$10^4$& 0.5 & 1.000  & -5.843 & -5.869 & -5.711 && 0.203 & 38.76 & 39.28 & 3.197 \\
		$10^5$ & 0.0 & 1.000 & -5.859 & -5.831 & -5.794 && 4.938 & 1314.6 & 1295.7 & 8.751 \\
		$10^5$& 0.5 & 1.000 & -5.817 & -5.826 & -5.798 && 4.652 & 1293.2 & 1283.4 & 37.84 \\
		$10^6$ & 0.0 & 1.000 & -5.803 & -5.809 & -5.693 && 9.582 & 3528.3 & 3600.0 & 123.4 \\
		$10^6$& 0.5 & 1.000 & -5.776 & -5.836 & -5.721 && 11.43 & 3445.7 & 3600.0 & 117.6 \\
		\hline
	\end{tabular}
	\label{tab:effect-highdim}
\end{table}

{\bf a) Effect of $m$.} We set ${(n,s,q,{\rm nf})=(1000,100,2,0)}$ and vary the number of measurements ${m\in\{250,300,\ldots,500\}}$. For each parameter setting $(m,n,s,q,{\rm nf})$, we perform 20 independent trials and report the median result. As illustrated in Figure \ref{fig:effect-m}, the recovery accuracy improves monotonically as $m$ increases. Notably, when $m\ge350$, all algorithms except NPGM successfully recover the true signal with 100\% accuracy. Meanwhile, APPA attains exact recovery in all cases.

{\bf b) Effect of $s$.} We fix ${(m, n, q, {\rm nf}) = (500, 1000, 2, 0)}$ and vary ${s \in \{250, 300, \ldots, 500\}}$. The median values over 20 trials are presented in Figure~\ref{fig:effect-s}. It is evident that the larger the value of $s$, the more difficult it becomes to detect the ground-truth signal. In all cases, APPA, L2ADMM, and GUROBI successfully achieve exact recovery, whereas EMADM and MEPM fail to do so when $s \ge 300$ and $s \ge 450$, respectively.

{\bf c) Effect of $\text{\rm nf}$.} We fix ${(m,n,s,q)=(500,1000,300,2)}$ and vary the noise level ${\rm nf}\in\{0,0.02,0.04,$ $\ldots,0.1\}$. Figure \ref{fig:effect-nf} presents the median recovery accuracy over 20 trials. The results demonstrate that APPA, MEPM, and L2ADMM exhibit robustness to noise. In contrast, GUROBI (within 600 seconds) fails to recover the signal when ${\rm nf}\ge 0.06$, while EMADM fails at ${\rm nf}= 0.1$. 

{\bf d) Effect of $q$.} We fix ${(m,n,s,\rm{nf})=(500,1000,500,0)}$ but increase $q$ over range $\{1.25,1.5,\ldots,2.5\}$. The median results over 20 trials are reported in Figure \ref{fig:effect-q}. As noted previously, GUROBI and EMADM are only applicable to the quadratic case ($q = 2$), and thus their results are omitted. Again, APPA maintains the highest accuracy, with the exception of scenarios where $q \le 1.5$. 

{\bf e) Effect of higher dimensions.} To evaluate performance on large-scale problems, we consider dimensions ${n\in\{10^4, 10^5, 10^6\}}$ with $(m, s) = (0.5n, 0.01n)$ and noise levels ${\rm nf} \in \{0, 0.5\}$. GUROBI and EMADM are excluded as their computational requirements become prohibitive when $n\geq 10^4$. For ${n\geq10^5}$, we employ sparse measurement matrices $\textbf{A}$ with $10^8$ nonzero entries. The median results over 10 independent trials are presented in Table~\ref{tab:effect-highdim}. Remarkably, APPA consistently outperforms all competing methods across all problem sizes, achieving superior recovery accuracy while maintaining the fastest computation time.

\subsection{MIMO  detection}
Two variants of the MIMO detection problem are considered in this section: classical MIMO detection, characterized by a quadratic objective function, and one-bit MIMO detection, which features a non-quadratic objective function.
\subsubsection{Classical MIMO detection}
The classical MIMO detection problem seeks to reconstruct an unknown signal $\w^*$ from noisy linear observations $\y=\H\w^*+\boldsymbol{\varepsilon}$. Here, $\H\in\C^{m\times n}$ denotes the channel matrix, $\y\in \C^m$ represents the observed signal, and $\boldsymbol{\varepsilon}\in \C^m$ is additive Gaussian noise with i.i.d. entries drawn from ${\cal N}(0,\varrho^2)$. The true signal $\w^*$ belongs to the constraint set $\Omega:=\{\w\in\C^n: {\rm Re}(\w)\in\{0,1\}^n,{\rm Im}(\w)\in\{0,1\}^n\}$, where ${\rm Re}(\w)$ and ${\rm Im}(\w)$ denote the real and imaginary components of $\w$, respectively. This detection problem can be formulated as a maximum likelihood estimation:

 \begin{eqnarray*}
 \min\limits_{\x\in{\R^{2n}}}~~ \frac{1}{2}\|\textbf{A}\x-\textbf{b}\|^2,~~~
{\rm s.t.}~~   \x\in \{0,1\}^{2n},
\end{eqnarray*}
where
\begin{eqnarray*}
\textbf{A}:= \left[
 \begin{array}{lr}
 {\rm Re}(\H)& -{\rm Im}(\H)\\
 {\rm Im}(\H)& {\rm Re}(\H)\\
 \end{array}
 \right],\qquad \textbf{b}:= \left[
 \begin{array}{ll}
 {\rm Re}(\y)\\
 {\rm Im}(\y)\\
 \end{array}
 \right],\qquad \x:= \left[
 \begin{array}{ll}
 {\rm Re}(\w)\\
 {\rm Im}(\w)\\
 \end{array}
 \right].
\end{eqnarray*}  

  \begin{figure}[!t]
 	\centering
 	
 	\begin{subfigure}[b]{0.49\textwidth}
 		\centering
 		\includegraphics[width=0.99\textwidth]{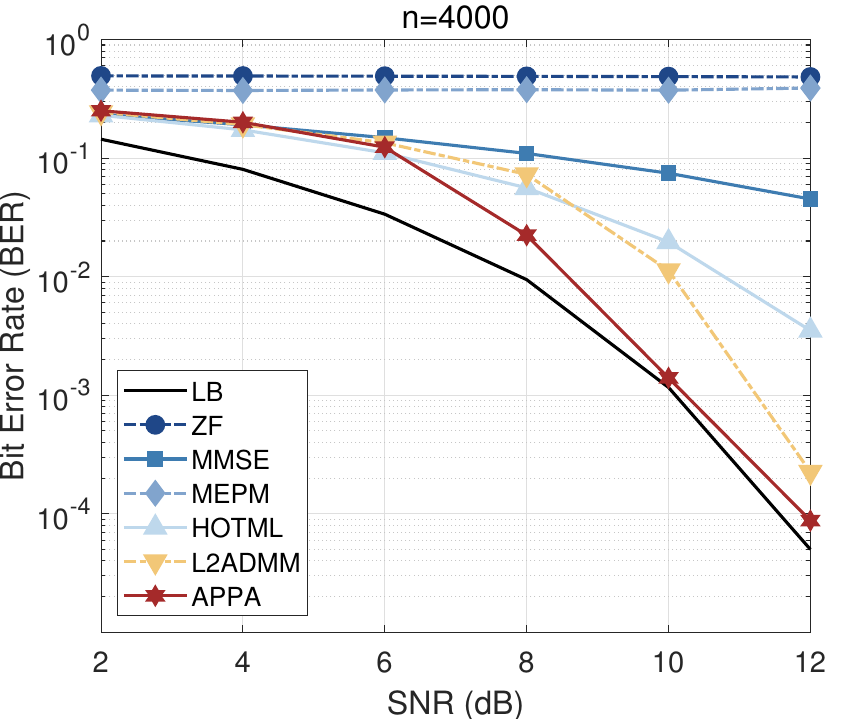}
 	\end{subfigure} 
 	\begin{subfigure}[b]{0.49\textwidth}
 		\centering
 		\includegraphics[width=0.99\textwidth]{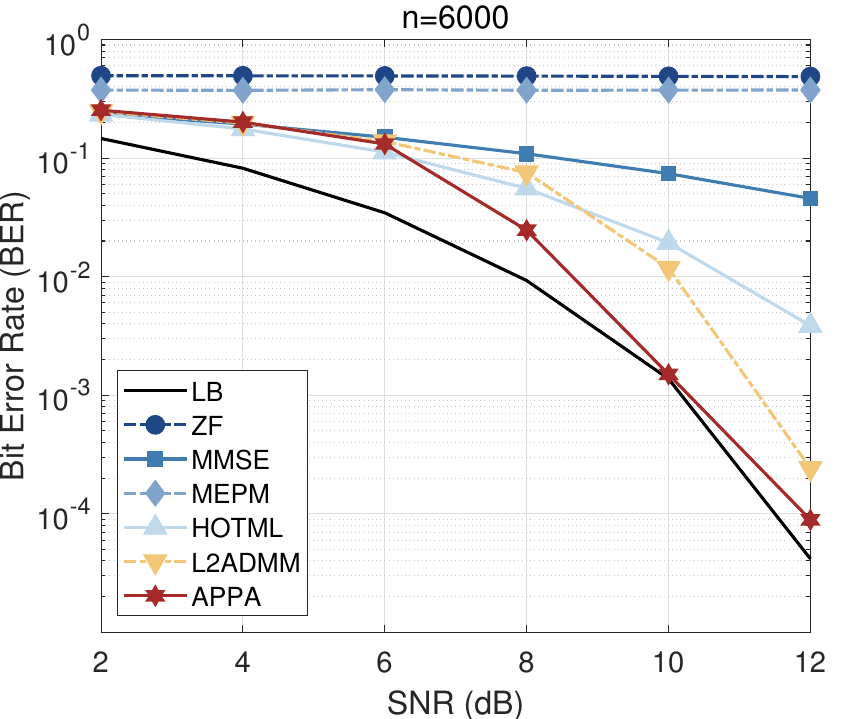}
 	\end{subfigure}
 	
 	\vspace{0.4cm}
 	
 	\begin{subfigure}[b]{0.49\textwidth}
 		\centering
 		\includegraphics[width=0.99\textwidth]{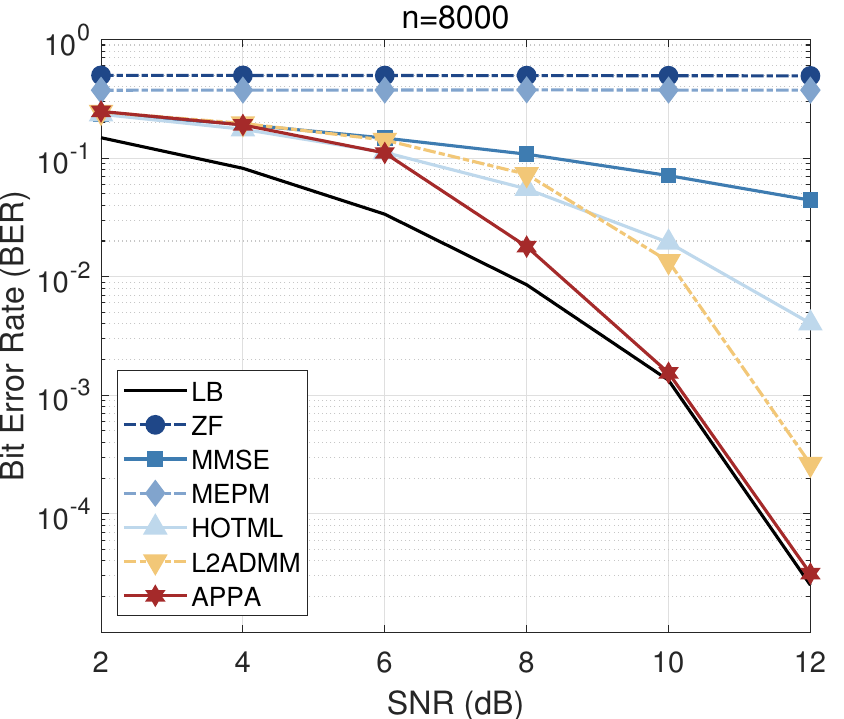} 
 	\end{subfigure} 
 	\begin{subfigure}[b]{0.49\textwidth}
 		\centering
 		\includegraphics[width=0.99\textwidth]{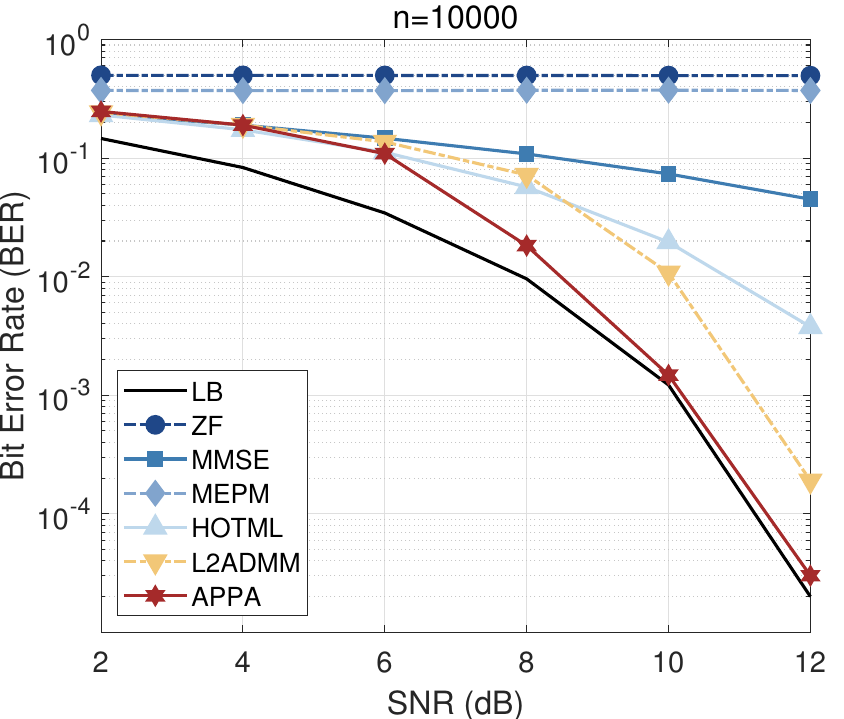}
 	\end{subfigure}
 	
 	\caption{Results on classical MIMO detection problems for i.i.d channels.}
 	\label{fig:BER-iid}
 \end{figure}
 
We consider two types of channel matrices $\textbf{H}$:
\begin{itemize}[leftmargin=15pt]
	\item \textbf{i.i.d. Gaussian channel:} The real and imaginary parts ${\rm Re}(\textbf{H})$ and ${\rm Im}(\textbf{H})$ are generated independently with entries sampled from the standard normal distribution, following the same procedure used for matrix $\textbf{A}$ in Section~\ref{sec:rec}.
	\item \textbf{Correlated MIMO channel:} Following \cite{loyka2001channel,shao2020binary}, we construct a spatially correlated channel as follows. Let $\widetilde{\textbf{H}}$ be an element-wise i.i.d. circularly symmetric complex Gaussian matrix with zero mean and unit variance. Define spatial correlation matrices $\textbf{R}\in\C^{m\times m}$ and $\textbf{T}\in\C^{n\times n}$ at the receiver and transmitter, respectively, with entries
	$$
	R_{ij}= \begin{cases}r^{i-j}, & \text{if } i \leq j, \\ R_{ji}^*, & \text{if } i > j,\end{cases}
	$$
	where ${|r| \leq 1}$. Matrix $\textbf{T}$ is constructed analogously. Performing Cholesky-like decompositions ${\textbf{R} = \textbf{P}\textbf{P}^*}$ and ${\textbf{T} = \textbf{Q}\textbf{Q}^*}$, the correlated channel matrix is given by $\textbf{H}=\textbf{P} \widetilde{\textbf{H}} \textbf{Q}$. We set $r = 0.2$ and $\textbf{R} = \textbf{T}$ in our experiments.
\end{itemize}
The transmitted signal $\textbf{w}^*$ is generated element-wise from a uniform distribution over the quaternary phase-shift keying (QPSK) constellation (see \cite{shao2020binary} for details). The noise variance $\varrho^2$ is determined by the signal-to-noise ratio (SNR), defined as
$$
\mathrm{SNR}=\mathbb{E}\|\textbf{H} \textbf{w}^*\|^2 / \mathbb{E}\|\boldsymbol{\varepsilon}\|^2.
$$
\vspace{-8mm}
 \begin{figure}[!t]
 	\centering
 	
 	\begin{subfigure}[b]{0.49\textwidth}
 		\centering
 		\includegraphics[width=0.99\textwidth]{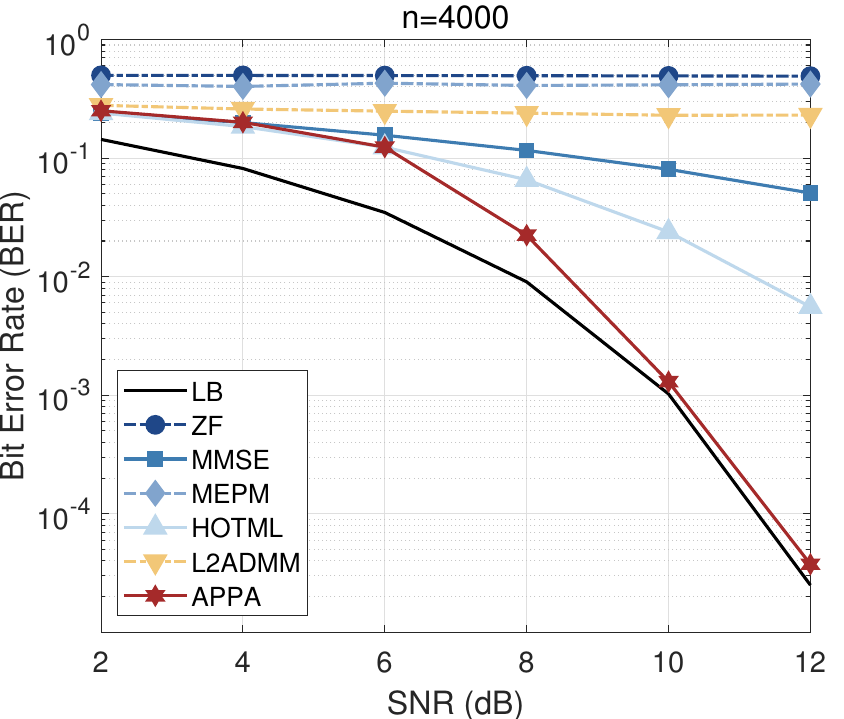}
 	\end{subfigure}
 	\begin{subfigure}[b]{0.49\textwidth}
 		\centering
 		\includegraphics[width=0.99\textwidth]{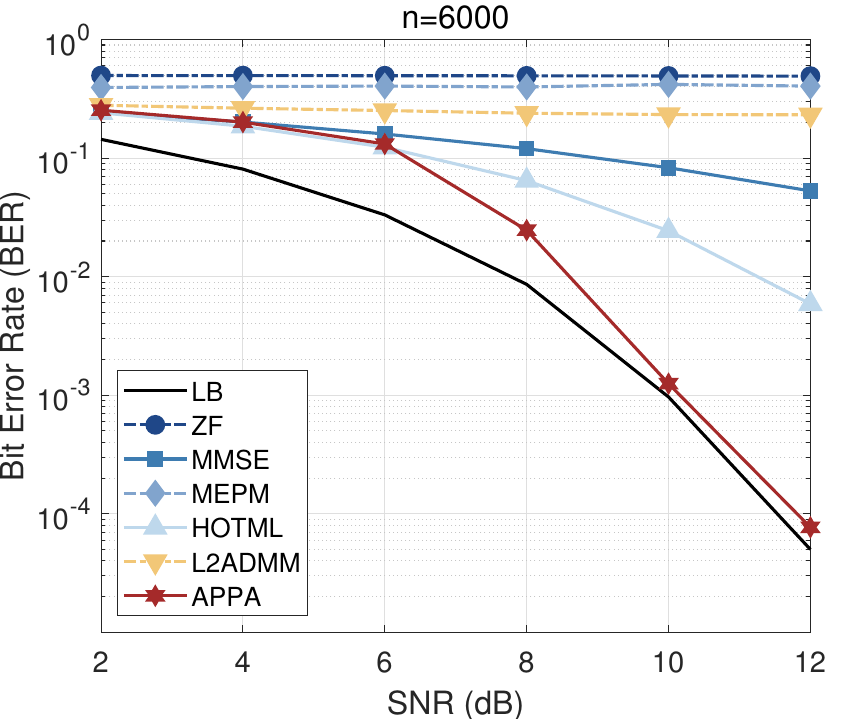}
 	\end{subfigure}
 	
 	\vspace{0.4cm}
 	
 	\begin{subfigure}[b]{0.49\textwidth}
 		\centering
 		\includegraphics[width=0.99\textwidth]{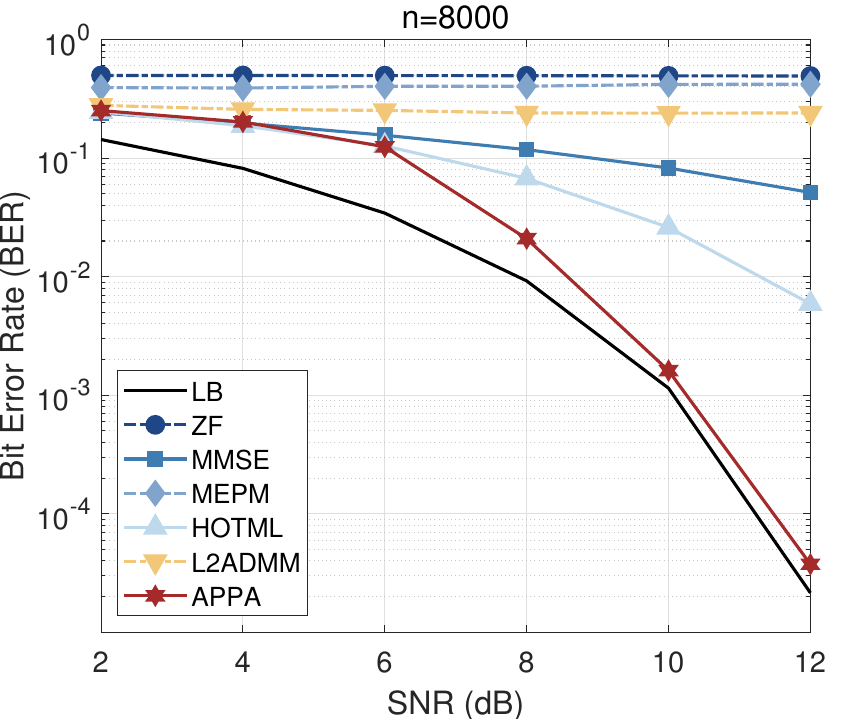} 
 	\end{subfigure}
 	\begin{subfigure}[b]{0.49\textwidth}
 		\centering
 		\includegraphics[width=0.99\textwidth]{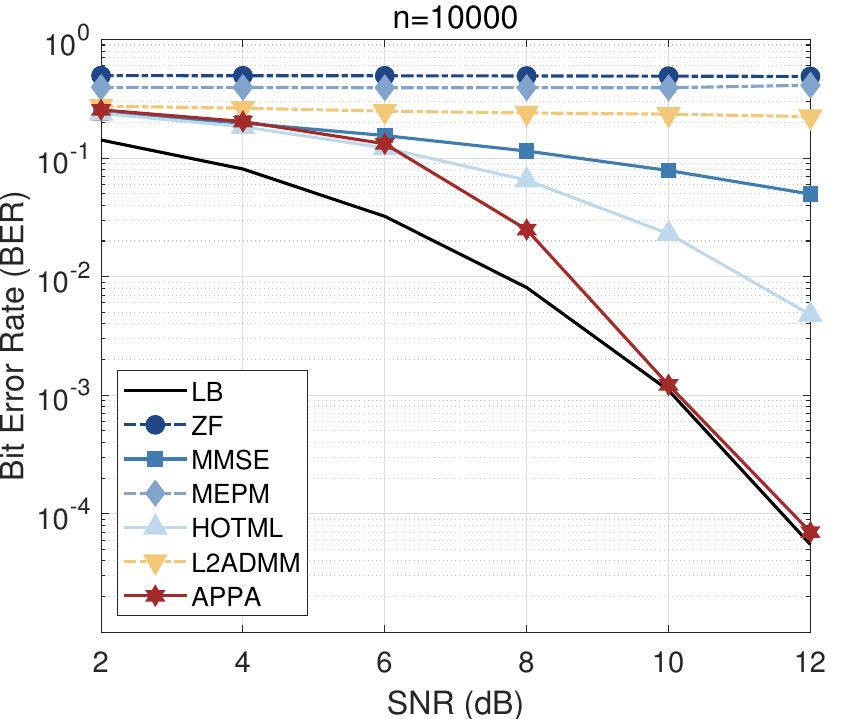}
 	\end{subfigure}
 	
 	\caption{Results on classical MIMO detection problems for correlated channels.} 
 	\label{fig:BER-cor}
 \end{figure}

 {\bf a) Benchmark methods.} We evaluate APPA against MEPM, L2ADMM, HOTML \cite{shao2020binary}, and two classical detectors: zero-forcing (ZF) and minimum-mean-square-error decision-feedback (MMSE-DF). A theoretical lower bound (LB) assuming no inter-signal interference is computed following \cite{shao2020binary}. The hyperparameters for all algorithms are summarized as follows. For APPA, we set $(\eta,\lambda,\pi,k_0)=(2,0.25,1.25,10)$ with initial penalty $\lambda_0=0.01\|\textbf{b}^{\top}\textbf{A}\|_{\infty}$; other parameters follow Section \ref{sec:rec}. MEPM and L2ADMM use identical settings to Section \ref{sec:rec}. For HOTML, we choose $\sigma_0=0.5$ and $\lambda_0=0.001$, with termination criteria: either $k>200$ iterations or $\|\lambda_k-\lambda_{k-1}\|\le 10^{-4}$. To ensure fair comparison, all algorithms start from $\x^0=\textbf{0}$. We assess detection accuracy through the bit error rate (BER), defined as
 $$
 \text{BER} = \frac{|\{i\in[2n]: x_i \neq x^*_i\}|}{2n},
 $$
 where $|\Omega|$ is the cardinality of $\Omega$. A smaller BER corresponds to more accurate signal recovery.
 
 
 \begin{table}[!t]
	\renewcommand{\arraystretch}{1.00}\addtolength{\tabcolsep}{6pt}
	\centering
	\caption{Average CPU time (in seconds) for classical MIMO detection problems, where A1-A4 stand for  APPA,   MEPM,   L2ADMM, and HOTML, respectively.}
 \begin{tabular}{c ccccc c ccc}
 			\hline
 			    &  \multicolumn{4}{c}{I.I.D channels} && \multicolumn{4}{c}{Correlated channels} \\
 			\cline{2-5} \cline{7-10}
 			SNR & A1 & A2 & A3 & A4  && A1 & A2 & A3 & A4   \\
 			\hline
 			&\multicolumn{9}{c}{$n=4000$}\\\hline
			 2 & 2.496 & 28.09 & 38.18 & 26.08 && 2.478 & 27.70 & 36.72 & 25.85 \\
			 4 & 2.498 & 28.17 & 38.22 & 26.12 && 2.625 & 28.42 & 34.68 & 25.86 \\
			 6 & 2.437 & 28.04 & 38.38 & 26.18 && 2.472 & 27.29 & 34.46 & 25.87 \\
			 8 & 2.289 & 27.94 & 38.47 & 26.40 && 2.354 & 27.99 & 34.39 & 25.85 \\
			 10 & 1.716 & 27.98 & 38.38 & 26.12 && 1.988 & 27.93 & 34.44 & 25.84 \\
			 12 & 1.763 & 27.48 & 33.89 & 26.42 && 1.911 & 27.72 & 35.24 & 25.92 \\
\hline
	&\multicolumn{9}{c}{$n=6000$}\\\hline
			 2 & 5.760 & 68.48 & 92.39 & 103.4 && 5.926 & 70.58 & 91.65 & 101.5 \\
			 4 & 5.852 & 68.74 & 92.60 & 103.3 && 5.856 & 70.24 & 86.82 & 101.5 \\
			 6 & 5.728 & 68.24 & 92.21 & 103.4 && 5.617 & 69.64 & 84.01 & 101.5 \\
			 8 & 5.532 & 68.21 & 92.66 & 103.8 && 5.599 & 70.35 & 84.27 & 101.5 \\
			 10 & 4.027 & 68.33 & 92.45 & 103.7 && 4.794 & 68.53 & 85.37 & 101.6 \\
			 12 & 3.823 & 68.08 & 83.09 & 103.8 && 3.976 & 69.68 & 86.70 & 101.5 \\
			\hline
				&\multicolumn{9}{c}{$n=8000$}\\\hline
			 2 & 12.25 & 134.1 & 180.7 & 251.6 && 12.36 & 131.9 & 177.5 & 261.9 \\
			 4 & 12.81 & 134.2 & 180.9 & 251.3 && 12.16 & 132.4 & 157.5 & 261.9 \\
			 6 & 12.34 & 134.1 & 181.8 & 251.3 && 12.09 & 129.9 & 157.3 & 261.8 \\
			 8 & 10.48 & 134.5 & 182.5 & 251.7 && 10.65 & 130.7 & 156.9 & 262.0 \\
			 10 & 8.994 & 133.0 & 179.9 & 253.1 && 8.716 & 127.6 & 159.7 & 261.8 \\
			 12 & 8.657 & 133.7 & 174.4 & 253.1 && 7.189 & 128.4 & 164.6 & 262.0 \\
			\hline
				&\multicolumn{9}{c}{$n=10000$}\\\hline
			 2 & 19.76 & 204.4 & 274.7 & 496.4 && 17.19 & 205.4 & 290.1 & 509.4 \\
			 4 & 19.31 & 204.2 & 275.3 & 496.7 && 17.84 & 204.6 & 263.0 & 509.6 \\
			 6 & 19.87 & 205.0 & 275.3 & 496.2 && 18.37 & 202.5 & 246.4 & 498.9 \\
			 8 & 16.59 & 204.4 & 277.8 & 496.5 && 15.65 & 206.2 & 244.0 & 509.1 \\
			 10 & 13.63 & 204.8 & 276.6 & 497.0 && 13.00 & 206.2 & 258.9 & 509.1 \\
			 12 & 13.32 & 200.9 & 265.0 & 497.1 && 12.29 & 200.5 & 251.4 & 509.0 \\
			\hline
		\end{tabular} 
	\label{tab:MIMO-time}
\end{table}

 {\bf b) Numerical comparison.}  We evaluate performance on critically determined systems (${m = n}$) with $n\in\{4000,6000,8000,10000\}$. The BER versus SNR curves for i.i.d. and correlated channels are shown in Figures~\ref{fig:BER-iid} and~\ref{fig:BER-cor}, respectively, where each data point represents the average of 20 trials. In the i.i.d. channel scenario, APPA delivers the best BER performance, with L2ADMM and HOTML ranking second and third, while ZF exhibits the poorest accuracy. Under correlated channels, APPA consistently achieves the minimum BER at most SNR levels, with HOTML as the second-best performer. Notably, L2ADMM fails to produce reliable results in the correlated setting. The corresponding computational time is reported in Table~\ref{tab:MIMO-time} for problem sizes ${n=10^3}$ and ${n=10^4}$. APPA exhibits superior speed across all test cases. As an illustrative example, for the correlated channel with ${n=10^4}$ and SNR = 12 dB, APPA requires only 13.32 seconds, while the other three methods each require over 200 seconds.

\subsubsection{One-bit MIMO detection}
The one-bit MIMO detection \cite{shao2020binary} aims to recover the transmitted binary signal $\textbf{z}\in\{-1,1\}^n$ from one-bit quantized observations $\textbf{y}\in\{-1,1\}^m$ according to the measurement model
\begin{equation}
	\label{eq:onebit}
	\textbf{y}=\operatorname{sgn}(\textbf{H} \textbf{z}+\textbf{v}),
\end{equation}
where $\operatorname{sgn}(\cdot)$ denotes the element-wise sign function that maps each component to $\{-1,1\}$, $\textbf{H} \in \mathbb{R}^{m \times n}$ represents the known channel matrix, and $\textbf{v} \in \mathbb{R}^m$ is additive Gaussian noise with i.i.d. entries drawn from ${\cal N}(0,\varrho^2)$. The noise variance $\varrho^2$ characterizes the quality of the communication channel. Under this observation model, the maximum-likelihood detector takes the form
\begin{equation}
	\label{eq:onebitprob}
	\min\limits_{\z\in\{-1,1\}^{n}}~~ -\sum_{i=1}^m \log \Phi\left(\frac{y_i \langle\textbf{h}_i, \z\rangle}{\varrho}\right),
\end{equation}
where $\Phi(t)=\int_{-\infty}^t \frac{1}{\sqrt{2 \pi}} e^{-\tau^2 / 2} d \tau$ is the cumulative distribution function (CDF) of the standard normal distribution, and $\textbf{h}_i$ denotes the $i$-th row of $\textbf{H}$. To reformulate problem \eqref{eq:onebitprob} into the standard form \eqref{UBIP}, we perform the variable transformation $\x=(\z+\mathbf{1})/2$, which maps the bipolar constraint set $\{-1,1\}^n$ to the binary set $\{0,1\}^n$. 

 {\bf a) Benchmark methods.} The benchmark algorithms include the zero-forcing (ZF) detector, MEPM, L2ADMM, and HOTML. For APPA, we set $(\eta,\lambda,\pi, k_0)=(0.1,0.5,1.2,10)$,  $\theta=\|\H\|_{\infty}+\|\y\|_{\infty}$, and $\lambda_0=0.005\|\y^{\top}\H\|_{\infty}$. For MEPM and L2ADMM, we set $L=\|\H\|_F^2/n$.  HOTML maintains the parameter settings used in the classical MIMO case. Its maximum number of iterations is set to 300. 

{\bf b) Numerical comparison.} We first investigate how the measurement-to-dimension ratio $m/n$ affects detection performance. Fixing ${\text{SNR} = 15}$ dB, we vary ${m/n \in \{1, 1.5, \ldots, 3\}}$ and test different problem dimensions. Figure~\ref{fig:BER-1bit-effmn} presents the BER averaged over 20 independent trials. The results show that APPA consistently delivers the lowest BER across all tested configurations.

	\begin{figure}[!t]
		\centering
		
		\begin{subfigure}[b]{0.49\textwidth}
			\centering
			\includegraphics[width=0.99\textwidth]{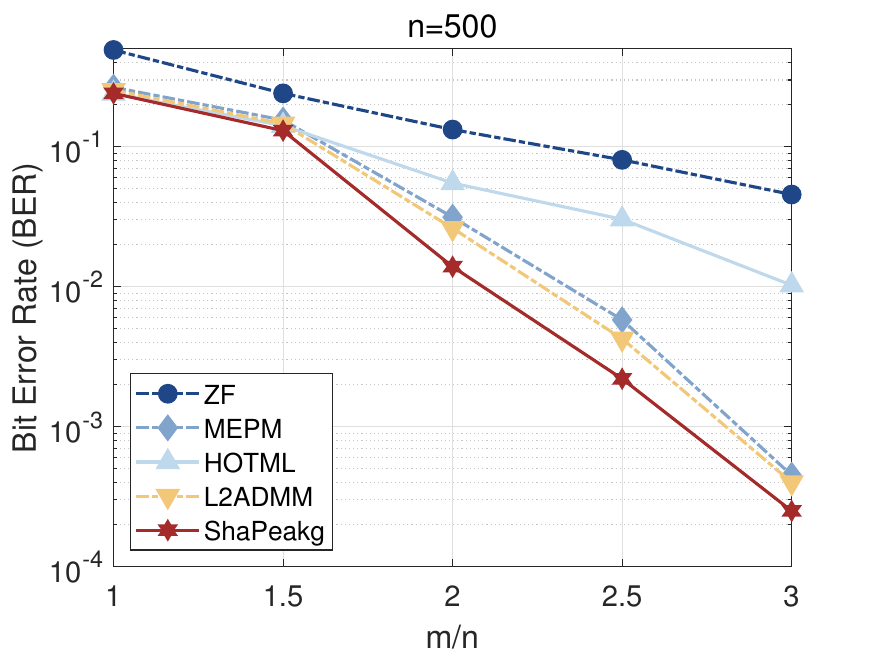}
		\end{subfigure} 
		\begin{subfigure}[b]{0.49\textwidth}
			\centering
			\includegraphics[width=0.99\textwidth]{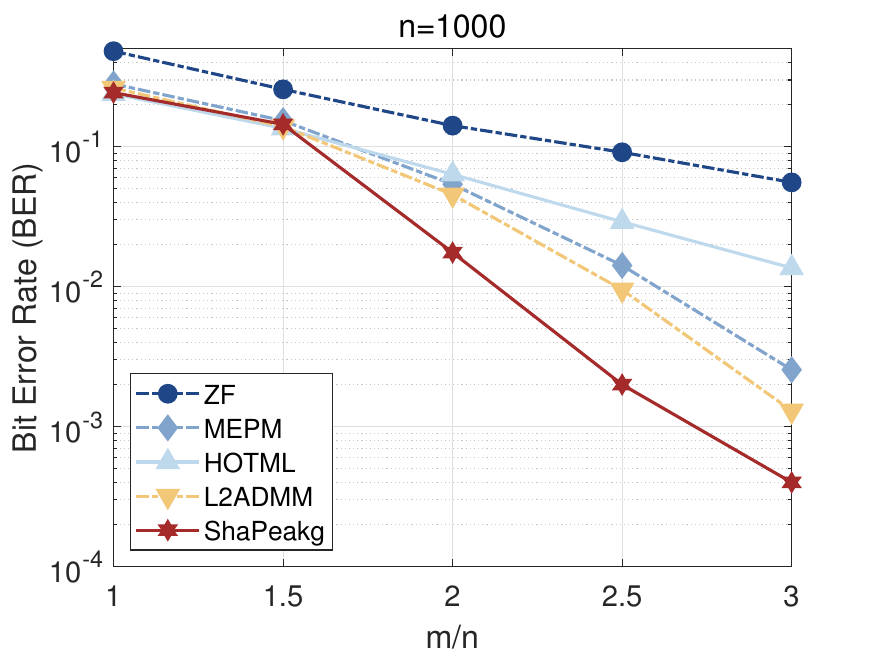}
		\end{subfigure}
		
		\caption{Effect of $m/n$ for one-bit MIMO detection problems.}
		\label{fig:BER-1bit-effmn}
	\end{figure}

	\begin{figure}[!t]
		\centering
		
		\begin{subfigure}[b]{0.49\textwidth}
			\centering
			\includegraphics[width=0.99\textwidth]{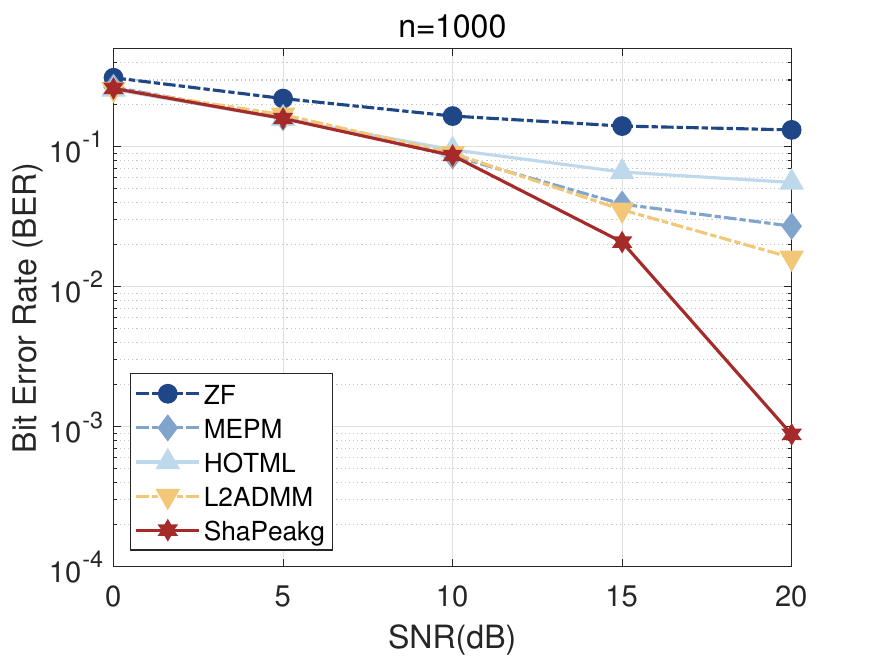}
		\end{subfigure} 
		\begin{subfigure}[b]{0.49\textwidth}
			\centering
			\includegraphics[width=0.99\textwidth]{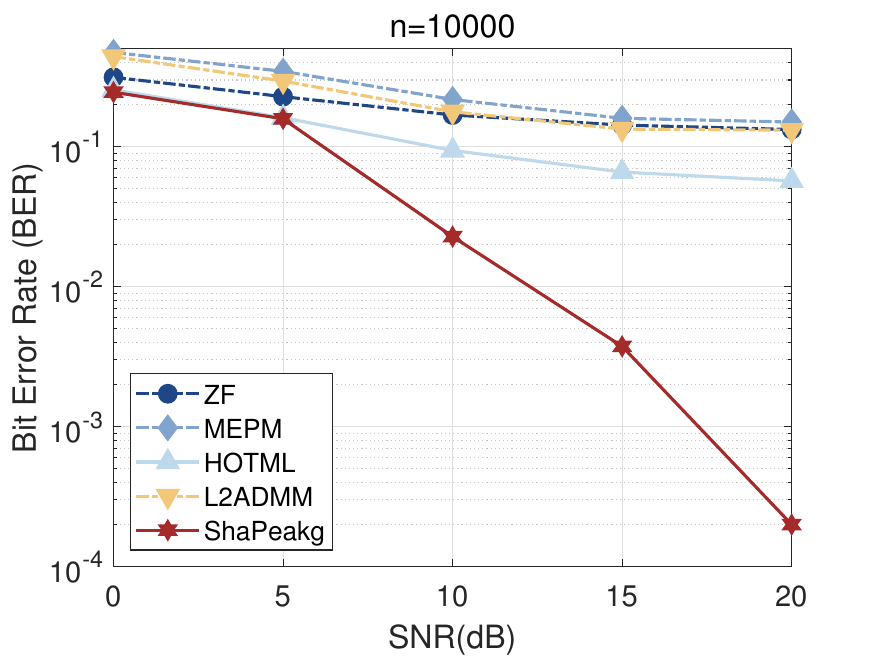}
		\end{subfigure}
		
%
%
		\caption{Effect of SNR for  one-bit MIMO detection problems.}
		\label{fig:BER-1bit}
	\end{figure}

	\begin{table}[!t]
		\renewcommand{\arraystretch}{1.0}\addtolength{\tabcolsep}{6pt}
		\centering
		\caption{Average CPU time (in seconds) for one-bit MIMO detection problems, where A1-A4 stand for  APPA,   MEPM,   L2ADMM, and HOTML, respectively.}
			\begin{tabular}{cccccccccc}
				\hline
			  	&\multicolumn{4}{c}{$m=1000, n=500$}&&\multicolumn{4}{c}{$m=10000, n=5000$}\\\cline{2-5}\cline{7-10}
			  	 SNR   & A1 & A2 & A3 & A4    & & A1 & A2 & A3 & A4   \\ \hline
			
				  0 & 0.024 & 0.609 & 0.856 & 0.024 && 1.356 & 142.9 & 188.4 & 4.021\\
				  5 & 0.030 & 0.649 & 0.912 & 0.033 && 1.547 & 143.4 & 196.8 & 6.502\\
				  10 & 0.081 & 0.673 & 0.934 & 0.069 && 4.399 & 139.8 & 196.3 & 13.88\\
				  15 & 0.105 & 0.678 & 0.994 & 0.088 && 7.287 & 144.0 & 195.7 & 14.28\\
				  20 & 0.120 & 0.693 & 1.017 & 0.087 && 8.919 & 143.0 & 197.2 & 14.45\\
				\hline
			\end{tabular}
		\label{tab:1bitMIMO-time}
	\end{table}

Next, we examine performance under varying noise levels by fixing ${m/n = 2}$ and testing ${\text{SNR}} \in \{0, 5$, $\ldots, 20\}$. The average BER over 20 runs is displayed in Figure~\ref{fig:BER-1bit}. At low SNR levels (${\text{SNR} \leq 10}$ dB), all algorithms exhibit poor performance, failing to reliably recover the transmitted signal. However, as SNR increases beyond 10 dB, APPA begins to outperform competing methods, with its superiority becoming increasingly evident at higher SNR values. Table~\ref{tab:1bitMIMO-time} reports the runtime comparison. APPA demonstrates substantial computational advantages across all scenarios. For example, in the case of ${n = 5000}$ and ${\text{SNR} = 20}$ dB, APPA completes in approximately 9 seconds, whereas L2ADMM requires nearly 200 seconds.

%

\subsection{QUBO}
	The QUBO problem takes the following form, 
 \begin{eqnarray}\nonumber
\min\limits_{\x\in{\R^n}}~ \frac{1}{2}\langle \x, \Q\x\rangle,~~~
{\rm s.t.}~~ \x\in \{0,1\}^n,
\end{eqnarray}
where ${\Q\in\R^n}$ is a symmetric matrix. We consider two types of test instances for evaluating algorithmic performance on QUBO problems. The first corresponds  standard benchmark problems from the Beasley instances \cite{wiegele2007biq}, which are widely used for evaluating QUBO solvers. For this instance, we test three problem sizes: dimensions $n \in \{100, 250, 500\}$. Each size category contains 10 instances with densities ranging from 0.2 to 0.8. The second corresponds to synthetically generated instances following \cite{chen2023montecarlopolicygradient}, where $\textbf{Q}$ is constructed with density 0.8 and nonzero entries uniformly distributed over $[10, 100]$ when $n<10^4$, and with density 0.005 when $n \geq 10^4$. 

The algorithms benchmarked in this experiment include  MEPM, L2ADMM, and GUROBI. The hyperparameters are given as follows. For APPA, we set $(\eta, \sigma, \lambda, \pi, k_0) = (1, 10^{-8}, 0.25, 1.5, 100)$, ${\theta = \|\textbf{Q}\|_{\infty}}$ and ${\lambda_0 = 0.001\|\textbf{Q}\|_F}$. For MEPM and L2ADMM, we set ${L ={1}/{n} \|\textbf{Q}\|_F^2}$, while other parameters remain consistent with those used in previous experiments. Default settings are adopted for GUROBI. To evaluate performance, we compute the relative optimality gap by
\begin{eqnarray*}\label{def-gap-obj}
 \text{Gap} = \frac{|\text{obj}-\text{lowest}|}{|\text{lowest}|} \times 100\%,
\end{eqnarray*} 
where `obj' denotes the objective value obtained by a specific algorithm, and `lowest' denotes the the latest best-known values (for Beasley benchmarks) or lowest objective value achieved among all compared methods (for synthetic instances).

{\bf a) Comparisons on Beasley instances.} Table~\ref{tab:beasley-results} presents the results on Beasley benchmark instances. In terms of solution quality, APPA achieves gaps second only to GUROBI across most instances, and even attains the optimal solution in several cases (e.g., bqp100-1, bqp100-10, bqp250-7). Regarding computational efficiency, APPA maintains the shortest computational time in the majority of instances, while GUROBI's runtime reaches 600 seconds for all problems with $n \geq 250$.

\begin{table}[!h]
	\renewcommand{\arraystretch}{0.96}\addtolength{\tabcolsep}{4pt}
	\centering
	\caption{Results on Beasley benchmark instances}
	\label{tab:beasley-results}
	\begin{tabular}{lccccccccccc}
		\hline
		& \multicolumn{2}{c}{APPA} & & \multicolumn{2}{c}{MEPM} & & \multicolumn{2}{c}{L2ADMM} & & \multicolumn{2}{c}{GUROBI} \\
		\cline{2-3} \cline{5-6} \cline{8-9} \cline{11-12}
		Instance & gap & time & & gap & time & & gap & time & & gap & time \\
		\hline
		bqp100-1 & 0.00 & 0.015 & & 2.31 & 0.053 & & 2.08 & 0.082 & & 0.00 & 0.420 \\
		bqp100-2 & 0.43 & 0.004 & & 0.45 & 0.037 & & 0.45 & 0.047 & & 0.00 & 0.337 \\
		bqp100-3 & 0.28 & 0.005 & & 0.28 & 0.034 & & 0.28 & 0.051 & & 0.28 & 0.474 \\
		bqp100-4 & 0.46 & 0.004 & & 0.60 & 0.035 & & 0.60 & 0.034 & & 0.46 & 0.422 \\
		bqp100-5 & 2.37 & 0.004 & & 0.74 & 0.036 & & 0.74 & 0.050 & & 0.00 & 0.398 \\
		bqp100-6 & 1.68 & 0.071 & & 1.97 & 0.046 & & 1.97 & 0.054 & & 0.90 & 0.499 \\
		bqp100-7 & 1.15 & 0.004 & & 1.47 & 0.048 & & 1.31 & 0.070 & & 0.00 & 0.405 \\
		bqp100-8 & 1.03 & 0.054 & & 1.63 & 0.055 & & 1.63 & 0.066 & & 0.00 & 0.326 \\
		bqp100-9 & 0.10 & 0.003 & & 0.10 & 0.034 & & 0.10 & 0.029 & & 0.00 & 0.330 \\
		bqp100-10 & 0.00 & 0.003 & & 0.14 & 0.044 & & 0.14 & 0.057 & & 0.00 & 0.327 \\
		\hline
		bqp250-1 & 0.62 & 0.046 & & 0.90 & 0.079 & & 0.88 & 0.119 & & 0.00 & 600.0 \\
		bqp250-2 & 0.85 & 0.003 & & 1.58 & 0.054 & & 1.58 & 0.052 & & 0.00 & 600.0 \\
		bqp250-3 & 0.24 & 0.040 & & 0.25 & 0.075 & & 0.25 & 0.092 & & 0.00 & 600.0 \\
		bqp250-4 & 0.38 & 0.002 & & 0.39 & 0.066 & & 0.39 & 0.073 & & 0.00 & 600.0 \\
		bqp250-5 & 0.38 & 0.006 & & 0.47 & 0.062 & & 0.47 & 0.091 & & 0.20 & 600.0 \\
		bqp250-6 & 0.28 & 0.003 & & 0.62 & 0.075 & & 0.62 & 0.084 & & 0.10 & 600.0 \\
		bqp250-7 & 0.00 & 0.002 & & 0.11 & 0.079 & & 0.11 & 0.112 & & 0.00 & 600.0 \\
		bqp250-8 & 4.11 & 0.044 & & 3.09 & 0.078 & & 2.76 & 0.116 & & 0.00 & 600.0 \\
		bqp250-9 & 0.56 & 0.021 & & 1.43 & 0.083 & & 1.29 & 0.128 & & 0.03 & 600.0 \\
		bqp250-10 & 0.21 & 0.003 & & 0.40 & 0.075 & & 0.40 & 0.081 & & 0.00 & 600.0 \\
		\hline
		bqp500-1 & 1.44 & 0.005 & & 2.25 & 0.185 & & 1.78 & 0.252 & & 0.06 & 600.0 \\
		bqp500-2 & 0.25 & 0.009 & & 0.51 & 0.192 & & 0.45 & 0.270 & & 0.00 & 600.0 \\
		bqp500-3 & 0.22 & 0.006 & & 0.59 & 0.156 & & 0.45 & 0.218 & & 0.00 & 600.0 \\
		bqp500-4 & 0.23 & 0.006 & & 0.39 & 0.150 & & 0.39 & 0.215 & & 0.00 & 600.0 \\
		bqp500-5 & 0.86 & 0.005 & & 1.27 & 0.150 & & 1.23 & 0.228 & & 0.12 & 600.0 \\
		bqp500-6 & 0.54 & 0.005 & & 2.38 & 0.165 & & 2.27 & 0.233 & & 0.00 & 600.0 \\
		bqp500-7 & 0.81 & 0.018 & & 1.85 & 0.171 & & 1.70 & 0.244 & & 0.08 & 600.0 \\
		bqp500-8 & 0.52 & 0.009 & & 1.26 & 0.166 & & 1.26 & 0.265 & & 0.00 & 600.0 \\
		bqp500-9 & 0.51 & 0.007 & & 1.14 & 0.161 & & 1.06 & 0.225 & & 0.00 & 600.0 \\
		bqp500-10 & 1.06 & 0.024 & & 0.69 & 0.155 & & 0.69 & 0.210 & & 0.00 & 600.0 \\
		\hline
	\end{tabular}
\end{table}

{\bf b) Comparisons on synthetic instances.} We first evaluate performance on small-scale problems, with results over 20 independent runs presented in Table~\ref{tab:qubo-gap-time-low}. Here, `best' and `mean' denote the minimum and average optimality gaps, respectively, among the 20 trials. Across all test cases, GUROBI achieves the smallest gap (i.e., the best objective values), with APPA as the second-best performer. However, APPA demonstrates a significant computational advantage, requiring substantially less time than competing methods.


\begin{table}[!t]
	\renewcommand{\arraystretch}{0.96}\addtolength{\tabcolsep}{-1pt}
	\centering
	\caption{Gap and time (in seconds) for low-dimensional instances} 
	\begin{tabular}{cccccccccccccccc}
		\hline
		& \multicolumn{3}{c}{APPA} && \multicolumn{3}{c}{MEPM} && \multicolumn{3}{c}{L2ADMM} && \multicolumn{3}{c}{GUROBI} \\ 
		\cline{2-4}\cline{6-8}\cline{10-12}\cline{14-16}
		$n$ 
		& best & mean & time 
		&& best & mean & time 
		&& best & mean & time 
		&& best & mean & time \\ \hline
		1000  & 0.39 & 0.80 & 0.057 && 2.17 & 2.82 & 0.352 && 1.92 & 2.39 & 0.462 && 0.00 & 0.00 & 600.0 \\
		3000  & 0.49 & 0.59 & 0.228 && 2.97 & 3.74 & 3.089 && 2.68 & 3.49 & 3.742 && 0.00 & 0.00 & 600.0 \\
		5000  & 0.39 & 0.53 & 1.041 && 3.86 & 4.56 & 22.05 && 3.42 & 4.15 & 25.73 && 0.00 & 0.00 & 600.0 \\
		7000  & 0.30 & 0.39 & 2.691 && 4.56 & 4.97 & 47.86 && 4.04 & 4.53 & 57.33 && 0.00 & 0.00 & 600.0 \\
		10000 & 0.28 & 0.37 & 6.082 && 5.43 & 5.76 & 112.3 && 4.88 & 5.32 & 129.8 && 0.00 & 0.00 & 600.0 \\ \hline
	\end{tabular}
	\label{tab:qubo-gap-time-low}
\end{table}

For large-scale experiments with ${n \in \{80000, 100000\}}$, we exclude GUROBI due to its excessive runtime requirements. Here, cases 1-5 represent five different proportions of nonnegative entries in $\textbf{Q}$: $\{0.4995, 0.4999, 0.49995, 0.49999, 0.5\}$, respectively. Table~\ref{tab:qubo-high} shows that APPA consistently obtains lower objective values and dramatically faster computation times compared to MEPM and L2ADMM. For example, in case 5 with ${n = 10^5}$, APPA completes in approximately 80 seconds, while MEPM and L2ADMM require 857 and 1083 seconds.


\begin{table}[!t]
	\renewcommand{\arraystretch}{1.0}\addtolength{\tabcolsep}{2pt}
	\centering
	\caption{Gap and time (in seconds) for higher dimensional instances}
	\begin{tabular}{ccccccccccccc}
		\hline
		& &  \multicolumn{3}{c}{APPA} && \multicolumn{3}{c}{MEPM} && \multicolumn{3}{c}{L2ADMM} \\ 
		\cline{3-5}\cline{7-9}\cline{11-13}
		$n$ & case & best & mean & time && best & mean & time && best & mean & time \\
		\hline
		\multirow{5}{*}{8e4} 
		& 1 & 0.00 & 0.00 & 47.50 && 1.53 & 1.63 & 613.7 && 1.29 & 1.36 & 768.5 \\
		& 2 & 0.00 & 0.00 & 49.38 && 1.56 & 1.65 & 613.5 && 1.29 & 1.37 & 780.8 \\
		& 3 & 0.00 & 0.00 & 44.67 && 1.58 & 1.64 & 605.8 && 1.30 & 1.36 & 770.4 \\
		& 4 & 0.00 & 0.00 & 52.11 && 1.66 & 1.69 & 613.8 && 1.36 & 1.40 & 760.3 \\
		& 5 & 0.00 & 0.00 & 43.05 && 1.63 & 1.67 & 607.7 && 1.36 & 1.39 & 765.4 \\
		\hline
		\multirow{5}{*}{1e5} 
		& 1 & 0.00 & 0.00 & 87.10 && 1.79 & 1.80 & 901.5 && 1.49 & 1.51 & 1133.8 \\
		& 2 & 0.00 & 0.00 & 67.87 && 1.69 & 1.80 & 899.8 && 1.69 & 1.50 & 1138.7 \\
		& 3 & 0.00 & 0.00 & 75.53 && 1.78 & 1.82 & 898.8 && 1.46 & 1.51 & 1125.5 \\
		& 4 & 0.00 & 0.00 & 77.02 && 1.76 & 1.80 & 873.3 && 1.47 & 1.50 & 1109.1 \\
		& 5 & 0.00 & 0.00 & 76.69 && 1.78 & 1.80 & 857.6 && 1.50 & 1.51 & 1083.6 \\
		\hline
	\end{tabular}
	\label{tab:qubo-high}
\end{table}


 
\section{Conclusion}


The proposed piecewise cubic function provides an exact continuous reformulation of binary constraints, transforming the combinatorially hard problem into a tractable continuous optimization framework. The key contribution is threefold: exact penalty theory with a solution-independent threshold, P-stationary points for optimality characterization, and the APPA algorithm with finite-iteration convergence guarantee. Extensive numerical experiments demonstrate that APPA achieves superior performance compared to other solvers across diverse problems. Future work includes the extension to constrained binary optimization and exploring applications to other applications.


 
\vskip 0.2in
\bibliographystyle{abbrv}
\bibliography{refs}

\end{document}